\documentclass[11pt]{article}
\usepackage{amsmath,amssymb,amsthm}
\usepackage{enumitem}
\usepackage[affil-it]{authblk}
\usepackage{a4wide}
\usepackage{color}
\usepackage{hyperref}
\usepackage{xfrac}
\usepackage{tikz}
\usepackage{caption} 
\usetikzlibrary{arrows}
\usepackage{authblk}
\usepackage{graphicx}
\usepackage{mathrsfs}

\usepackage{multirow}%

\usepackage[title]{appendix}%
\usepackage{textcomp}%
\usepackage{manyfoot}%
\usepackage{booktabs}%
\usepackage{algorithm}%
\usepackage{algorithmicx}%
\usepackage{algpseudocode}%
\usepackage{listings}%
\usepackage{mathtools}

\newtheorem{thm}[equation]{Theorem}
\newtheorem{cl}[equation]{Claim}
\newtheorem{que}[equation]{Question}
\newtheorem{cor}[equation]{Corollary}
\newtheorem{prop}[equation]{Proposition}
\newtheorem{lemma}[equation]{Lemma}
\theoremstyle{definition}
\newtheorem{defn}[equation]{Definition}
\newtheorem{remark}[equation]{Remark}

\numberwithin{equation}{section}

\newcommand\degg{\, \mathsf{deg} \,}
\newcommand\spann{\, \mathsf{span}}

\DeclareMathOperator{\id}{id}

\title{Into Multiplier Hopf ($*$-)Graph Algebras}

\author[1,2,*]{Farrokh Razavinia\thanks{Supported by the Azarbaijan Shahid Madani University under grant No. 117.d.22844 - 08.07.2023} \thanks{This research was in part also supported by a grant from IPM (No. 1403170014)}
}
\affil[1]{Department of Mathematics, Azarbaijan Shahid Madani University, Tabriz 5375171379, Iran}
\affil[2]{School of Mathematics, Institute for Research in Fundamental Sciences (IPM), P.O. Box: 19395‐5746, Tehran, Iran}
\affil[]{f.razavinia@phystech.edu}

\date{}

\begin{document}

\maketitle

\begin{abstract}
This paper is concerned with the structures introduced recently by the authors of the current paper concerning the multiplier Hopf $*$-graph algebras and also the Cuntz-Krieger algebras and their relations with the $C^*$-graph algebras, and once again by using the $C^*$-graph algebra constructions associated to our toy example, to initiate our first class of examples concerning the multiplier Hopf $*$-graph algebras.

At the final part of the paper, we apply our study to the $SL(n)$ case over the field of complex numbers, and prove that $(\mathcal{O}(SL(n),\Delta)$ possesses the initial requirements of being a discrete quantum group in the sense of Van Daele, and propose a direction in approaching one step further to the problem raised by Wang, asking ``if finite groups of Lie type have an analogue of $q$-deformations into finite quantum groups?''.
\newline\newline
\textbf{MSC Numbers (2020)}: Primary 46L05; Secondary 46L67, 81P40, 46L55, 17B81.
 \hfill \newline
\end{abstract}

\section{Notations}
Throughout this paper, $\mathbb K$ will stand for the ground field, and it will be assumed to be arbitrary unless otherwise stated.

We often will simply denote the matrix algebra $M_d(\mathbb C)$ of $d\times d$ complex-valued matrices, with $M_d$, and the identity $d\times d$ matrix with $I_d$, and $M_q(n)$  will stand for the space of $n\times n$ quantum matrices, and by $\mathbb K[M_q(n)]$ we mean the ring of coordinate ring of $M_q(n)$.

$\mathcal{G}_n$ will stand for the directed locally finite connected graph associated with the defining relations of $\mathbb K[M_q(n)]$, and $\Pi_n$    represents the adjacency matrix of $\mathcal{G}_n$. After that, we will consider $\mathcal{G}_n:=\mathcal{G}(\Pi_n)$.

By $u=(u_{ij})_{i,j}$, we mean the magic unitary matrix and $\pi_n$ will stand for the unique commuting magic unitary matrix with $\Pi_n$.

For a Hilbert space $\mathcal{H}$, we will consider the set of all (bounded) linear operators from $\mathcal{H}$ to $\mathbb K$, with $B(\mathcal{H},\mathbb{K})$, and we write $B(\mathcal{H})\equiv B(\mathcal{H},\mathcal{H})$ for the set of all linear operators acting on $\mathcal{H}$.

We use $*$ to present the complex conjugate transpose of vectors and matrices. 

We will use $\overset{\leftarrow}{\degg}_{hk}$ in order to show the entry degree of the vertex $e_{hk}$, and by entry degree we mean the number of edges that are entering to the vertex $e_{hk}$, and analogously we have exit degree of the vertex $e_{hk}$, which will be considered by $\overset{\rightarrow}{\degg}_{hk}$.

\section{Introduction}\label{S:intro}
Going back to the mid-eighties, there is the fascinating constructive definition and the only exact and fulfilled definition of a category of quantum groups, known as the compact (matrix) quantum groups studied and developed by Woronowicz \cite{Wor87}, started from a compact (matrix) group $G$ and followed by considering the algebra of continuous complex-valued functions $C(G)$, equipped with a commutative Hopf algebra structure with respect to the pointwise multiplication, and making it into a commutative $C^*$-algebra, in fulfilment of generalizing the concept of a (compact) group, defined as follows:

A compact quantum group is a pair $G=(A,\Delta)$, for $A$ a $C^*$-algebra and $\Delta$ a unital $*$-homomorphism from $A$ to $A\otimes A$, called comultiplication, satisfying the coassociativity relation $(\Delta\otimes id)\circ\Delta=(id\otimes\Delta)\circ\Delta$, and the cancellation properties
\begin{align*}
    &\Delta(A)(1\otimes A)=\spann\{\Delta(a)(1\otimes b) | a,b\in A\}\\& \Delta(A)(A\otimes 1)=\spann\{\Delta(a)(b\otimes 1) | a,b\in A\},
\end{align*}
dense in $A\otimes A$, due to Woronowicz.

Our concern is more about the compact quantum groups arising from the (semi-)group algebras. For the compact group $G$, we can see $\mathbb CG$ as the group $C^*$-algebra associated with $G$, consisting of the set of finite linear combinations $\sum_{g\in G}c_gg$, for $c_g\in\mathbb C$, with the multiplication adopted from the group multiplication and equipped with the involution $\left(\sum c_gg\right)^*:=\sum \overline{c}_gg^{-1}$, isomorphic with the universal $C^*$-algebra
$$C^*\left(c_g | c_g ~\text{unitary}, ~c_gc_h=c_{gh}, c_{g}^{*}=c_{g^{-1}}\right).$$

Following the above discussion, and in order to provide an answer to a question by Connes, asking if ``there are quantum permutation groups, and what would they look like? \cite{Ban20}'', in late nineties, Wang came with an answer, saying that ``the quantum permutation group $S_{n}^{+}$ could be defined as the largest compact quantum group acting on the set $\{1,\ldots, N\}$ \cite{Ban20},'' by looking at it as the compact set $X_N:=\{x_1,\ldots, x_N\}$ consisting of a finite set of points (pointwise isomorphic) and studying its function space $C(X_N)\equiv C^*\left(p_1,\cdots,p_N ~\text{projections}~ | \sum_{i=1}^{N}p_i=1\right)$. This has led him to define the following
$$C(S_{n}^{+}):=C^*\left( u_{ij}, {i,j=1,\cdots,n}\mid u_{ij}=u_{ij}^{*}=u_{ij}^{2}, \sum_{k=1}^{n}u_{kj}=\sum_{k=1}^{n}u_{ik}=1 \right),$$
and calling $S_{n}^{+}=(C(S_{n}^{+}),u)$ the quantum symmetric (permutation) group as the quantum automorphism group of $X_N$, and proving that it satisfies the relations of being a compact (matrix) quantum group in the sense of Woronowicz. The main ingredients in defining $C(S_{n}^{+})$, meaning that the $u_{ij}$s, are very important to us in our construction of the ($*$-)multiplier Hopf graph algebras, and we have the following definition:
\begin{defn}\label{Def:MUM}
   Matrix $u=(u_{ij})_{i,j}$ with entries $u_{ij}$s from a non-trivial unital $C^{*}$-algebra satisfying relations $u_{ij}=u_{ij}^{*}=u_{ij}^{2}$ and $\sum_{k=1}^{n}u_{kj}=\sum_{k=1}^{n}u_{ik}=1$, will be called a magic unitary.
\end{defn}
\begin{remark}
    \begin{enumerate}
        \item Entries of the magic unitary matrix are projections, and all the distinct elements of the same row or same column are orthogonal, and sums of rows and columns are equal to 1. 
        \item A magic unitary matrix $u$ is orthogonal, meaning that we have $u=\overline{u}$ and $uu^t=I_n=u^tu$.
		\item If an element $p$ satisfies $p^2=p^{*}=p$, then it will be called a projection, and two projections will be orthogonal if we have $pq=0$, and a partition of the unity is a finite set of mutually orthogonal projections, and sum up to 1.
    \end{enumerate}
\end{remark}
Recently, in \cite{RV22}, Rollier and Vaes put one step forward and applied the above constructions to the connected locally finite graphs, and more than that, they used this association in order to make a bridge between the already known abstract concept of the multiplier Hopf algebras, introduced and studied by Van Daele \cite{VD94}, to a more intuitive field of Graph Theory, in the form of the following theorem
\begin{thm}\cite{RV22}\label{Th:RV}
    For $\Pi$ a locally connected finite graph with vertex set $I$, there exist a (necessarily unique) universal nondegenerate $*$-algebra $\mathcal{A}$ generated by the elements $u_{ij}$ satisfying the relations of the magic unitary matrix in Definition \ref{Def:MUM}, and a unique nondegenerate $*$-homomorphism $\Delta:\mathcal{A}\to M(\mathcal{A}\otimes \mathcal{A})$ taking $u_{ij}$ to $\sum_{k\in I}(u_{ik}\otimes u_{kj})$ for all $i,j\in I$, such that the pair $(\mathcal{A},\Delta)$ is a multiplier Hopf $*$-algebra in the sense of \cite{VD94}. 
\end{thm}
Following theorem \ref{Th:RV}, in \cite{RH24}, by using the $(n^2-2)$-connected locally finite graphs $\mathcal{G}_i=\{ \mathcal{G}(\pi_i)\mid i\in\{1,\cdots n\} \}$ associated with the adjacency matrices $\Pi_i$ of the coordinate ring of the quantum matrix algebras $M_q(n)$ and their commuting matrices $\pi_i$, for $i\in I=\{1,\cdots,n\}$, we showed that the set $\mathcal{G}_i$ possesses a nondegenerate $*$-monoid algebra structure equiped with the following binary operations
\begin{align}
	& \pi_i+\pi_j:=\left(V_i\cup V_j, E_i\cup E_j \right)\notag \\& \pi_i\to \pi_j:=\left(V_i\cup V_j, E_i\cup E_j \right),\label{GA:1}
	\end{align}
and the identity element $\pi_2$, and the  diagrammatic illustration (\ref{fig1}).


\begin{figure}
\begin{tikzpicture}[font=\sffamily]
\foreach \X [count=\Y starting from 2] in {$\mathcal{G}(\pi_2):=\mathcal{G}_2$,$\mathcal{G}(\pi_3):=\mathcal{G}_3$,$\mathcal{G}(\pi_4):=\mathcal{G}_4$,$\cdots\cdots\cdots$,$\mathcal{G}(\pi_n):=\mathcal{G}_n$}
{\draw (-\Y,-\Y/2) circle ({1.25*\Y} and \Y);
\node at (1-2*\Y,-1.1*\Y) {\X}; }
\end{tikzpicture}
    \caption{Illustration of the set of $n-1$ graphs $\mathcal{G}_i$}
    \label{fig1}
\end{figure}
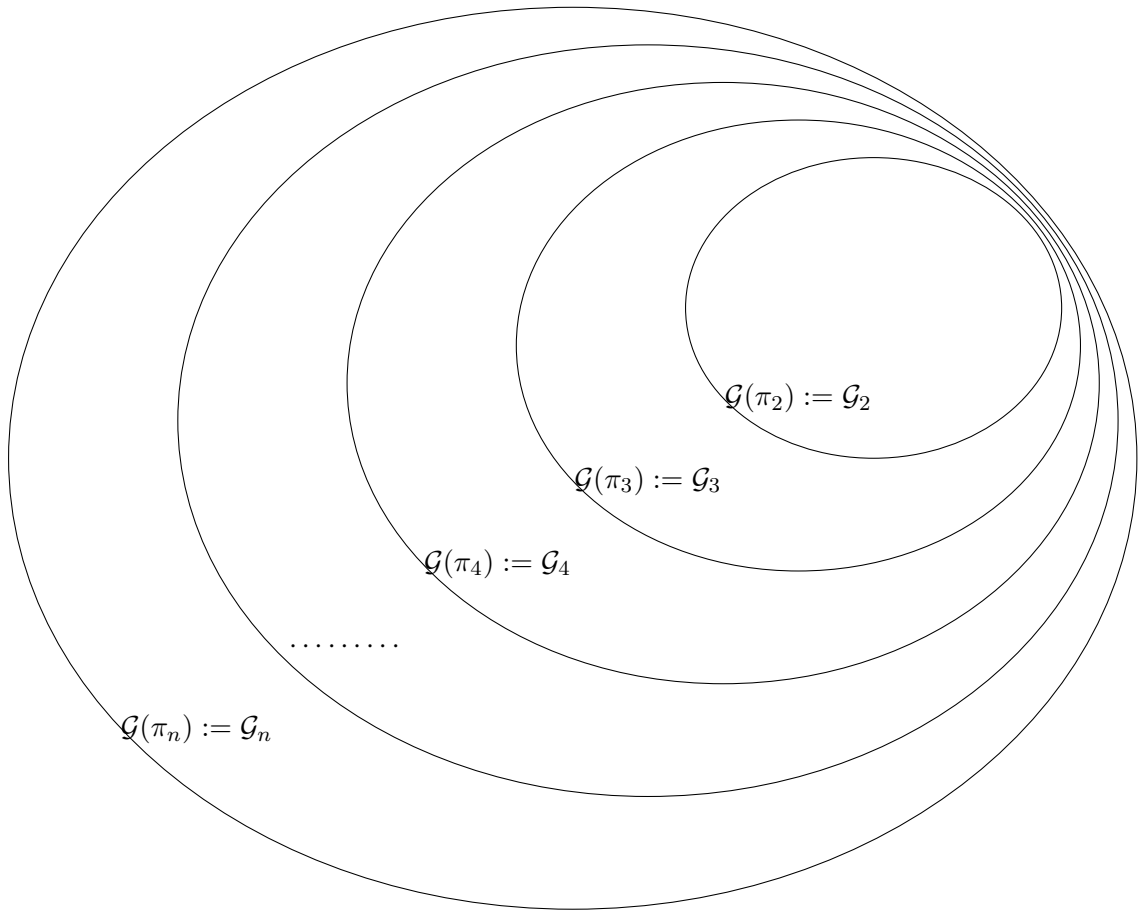



We refer the interested reader to \cite[Section 9]{RH24} for more information regarding the set of graphs $\mathcal{G}_i$, and quantum permutation groups $S_{n}^{+}$.

\subsection{Cuntz-Krieger Algebras}
Here in this paper, our developments are mostly concerned with the finite directed graphs. Talking about the Cuntz-Krieger algebras, we need to take some steps backwards and talk about the Cuntz algebras were introduced by Cuntz in 1977 as a class of simple purely infinite $C^*$-algebras generated by isometries ($S$ is an isometry if $S^*S=\id$) satisfying certain relations \cite{C77}.
\begin{defn}
  Let $n\geq2$. A universal $C^*$-algebra generated by isometries $S_1, S_2, \cdots, S_n$ satisfying in the following relations
   $$S_{1}^{*}S_1=S_{2}^{*}S_2=\cdots=S_{n}^{*}S_n=\id, \qquad \ \sum_{i=1}^{n}S_iS_{i}^{*}=\id,$$
   will be called a Cuntz algebra and will be denoted by $\mathcal{O}_n$.
\end{defn}

Following the above definition, to a directed graph $\Gamma$, one can associate a $C^*$-algebra $C^*(\Gamma^0,\Gamma^1):=C^*(\Gamma)$ by associating to its set of edges $\Gamma^1$ a set of partial isometries and to its set of vertices $\Gamma^0$ a set of pairwise orthogonal projections (Hilbert spaces) satisfying in some specific relations, studied first by Cuntz and Krieger in 1980 \cite{CK80}, as a generalization of the Cuntz algebras.

By a directed graph $\Gamma=(\Gamma^0,\Gamma^1)$ we mean a set $\Gamma^0$ of vertices and a set $\Gamma^1$ of edges alongside with source and range maps $r,s:\Gamma^1\to\Gamma^0$, and accordingly a vertex $v\in \Gamma^0$ will be called a sink if and only if $s^{-1}(v)$ is empty, meaning that when $v$ emits no edges, and a source if and only if $r^{-1}(v)$ is empty, meaning that when $v$ receives no edges.
\begin{remark}
    \begin{enumerate}
        \item The graph $\Gamma$ will be called simple if the map $\Gamma^1\to\Gamma^0\times\Gamma^0$ taking $e$ to $(s(e),r(e))$ is injective. Meaning that it will be called simple if consists of no multiple edges.
        \item The directed graph $L\Gamma=(L\Gamma^0,L\Gamma^1)$ with $L\Gamma^0=\Gamma^1$ and $L\Gamma^1=\{(e_1,e_2) | r(e_1)=s(e_2)\}$ as a subset of $\Gamma^1\times\Gamma^1$ will be called the line graph of $\Gamma$.
        \item Any line graph is simple.
        \item The $\Gamma^0\times\Gamma^0$-matrix $A_{\Gamma}(e)=\begin{cases}
            1 \quad e\in \Gamma^1,\\
            0 \quad e\notin \Gamma^1,
        \end{cases},$  will be called the adjacency matrix of $\Gamma=(\Gamma^0,\Gamma^1)$.
        \item The $\Gamma^1\times\Gamma^1$-matrix $E_{\Gamma}(e_1,e_2)=\begin{cases}
            1 \quad r(e_1)=s(e_2),\\
            0 \quad else,
        \end{cases},$  will be called the edge matrix of $\Gamma=(\Gamma^0,\Gamma^1)$.
        \item The edge matrix $E_{\Gamma}$ of $\Gamma$ equals the adjacency matrix $A_{L\Gamma}$ of $L\Gamma$.
    \end{enumerate}
\end{remark}
Before preparing for the main part of the paper, let us consider some preliminary definitions.
\begin{defn}\label{Def:PI}
    An $n\times n$ matrix $A$ will be called a partially isometric matrix (partial isometry), if $AA^*A=A$ satisfies.
\end{defn}
\begin{defn}\label{Def:OP}
    An $n\times n$ matrix $P$ will be called an orthogonal projection, if $P$ simultaneously is Hermitian and idempotent, meaning that if we have $P=P^*$ and $P^2=P$.
\end{defn}
\begin{remark}\cite{GMR19}
    \begin{enumerate}
        \item It is easy to see that $A$ is a partial isometry if and only if $A^*$ is a partial isometry. 
        \item It is not too difficult to see that if $P$ is an orthogonal projection, then $P$ is a partial isometry.
        \item Any unitary matrix is a partial isometry, and any invertible partial isometry is unitary.
        \item We have that  $A$ is a partial isometry if and only if $A^*A$ and $AA^*$ are orthogonal projections.
    \end{enumerate}
\end{remark}

Let us now recall what we mean by a graph $C^*$-algebra of a finite directed graph $\Gamma=(\Gamma^0,\Gamma^1)$.

For a finite directed graph $\Gamma$, and a finite or infinite dimensional Hilbert space $\mathcal{H}$, the set of mutually orthogonal projections $p_v\in\mathcal{H}$ for all $v\in\Gamma^0$ together with partial isometries $s_e\in\mathcal{H}$ for all $e\in\Gamma^1$ satisfying the relations

\begin{enumerate}\label{Equ:S}
    \item $s_{e}^{*}s_e=p_{r(e)}$ for all edges $e\in\Gamma^1$,\label{Equ:S:1}
    \item $p_v=\sum_{s(e)=v}s_es_{e}^{*}$ for the case when $v\in\Gamma^0$ is not a sink,\label{Equ:S:2}
\end{enumerate}
will be called a Cuntz-Krieger $\Gamma$-family in $C^*$-algebra $\mathcal{C}$, and we have the following definition:
\begin{defn}\cite{BEVW22}\label{Def:GCA}
    For finite directed graph $\Gamma=(\Gamma^0,\Gamma^1)$, the graph $C^*$-algebra $\mathcal{C}^*(\Gamma)$ is the universal $C^*$-algebra generated by a Cuntz-Krieger $\Gamma$-family $\{P_v,S_e\}$.
\end{defn}
\begin{remark}
    By a universal $C^*$-algebra in Definition \ref{Def:GCA}, we mean that for any other given Cuntz-Krieger $\Gamma$-family consisting of projections $p_v$ and partial isometries $s_e$ in $C^*$-algebra $\mathcal{C}$, there exists a unique $*$-homomorphism $\psi:C^*(\Gamma)\to B(\mathcal{H})$ such that $\psi(P_v)=p_v$ and $\psi(S_e)=s_e$.
\end{remark}
Now we are prepared to consider those $n\times n$ matrices with entries in $\{0,1\}$. By considering the relations \ref{Equ:S} of a Cuntz-Krieger $\Gamma$-family in $\mathcal{H}$, of partial isometries $s_i$ and orthogonal projections $p_i$, for $i\in\{1,\cdots,n\}$, with mutually orthogonal ranges, we will have the following Cuntz-Krieger relations
\begin{equation}\label{Equ:KC}
    s_{i}^{*}s_i=\sum_{j=1}^{n}a_{ij}s_js_{j}^{*}.
\end{equation}

Note that we will only consider the nondegenerate $*$-algebras, and hence $C^*$-algebras. We have the following Definition:
\begin{defn}
    For $n\times n$ matrix $\Pi\in M_n(0,1)$, the Cuntz-Krieger algebra $\mathcal{K}_{\Pi}$ will be defined as the (nondegenerate) $C^*$-algebra generated by a universal Cuntz-Krieger $\Gamma$-family $S_i$ for $i\in\{1,\cdots,n\}$ satisfying in \ref{Equ:KC}.
\end{defn}
\begin{remark}
    In the literature, in order to define a Cuntz-Krieger algebra the assumption of working with a nondegenerate matrix (having no sources and sinks) is assumed essential. But here in this paper, we won't make any further hypotheses on our matrices. So, it could be for instance that all the entries are zero, and hence in that case the algebra we get will be very degenerate and not good in order to work with concerning the multiplier Hopf algebras!
\end{remark}
\section{Graph $C^*$-algebras}

Going back to our construction made in \cite[Proposition 9]{RH24}, also Proposition \ref{Prop:mHA:graph} in the current paper, we can associate the following degenerate Cuntz-Krieger algebra $\mathcal{K}_{\Pi_n}$ to the graphs $\Pi_n$ associated with the coordinate ring $\mathbb K[M_q(n)]$. But before we get into this problem, let us recall a very important fact concerning the Cuntz-Krieger algebras and their relations. In what follows let $\mathcal{H}$ be as before.
For any finite locally connected (directed) graph $E=(E^0,E^1)$, it is well known that we have $P_v\mathcal{H}=\left(\sum_{\{e\in E^1\mid r(e)=v\}}S_eS_{e}^{*}\right)\mathcal{H}=\oplus_{\{e\in E^1\mid r(e)=v\}}S_e\mathcal{H}$.

Now let us consider the very initial part of our toy example. For graph $\mathcal{G}(\Pi_2)$ associated with the coordinate ring of $M_q(2)$, consider its set of vertices and edges as $\mathcal{G}^0=\{x_{11}:=u, x_{12}:=v, x_{22}:=k, x_{21}:=w\}$ and $\mathcal{G}^1=\{x_{11}\overrightarrow{\sim}x_{12}:=e, x_{11}\overrightarrow{\sim}x_{21}:=f, x_{12}\overrightarrow{\sim}x_{22}:=h, x_{21}\overrightarrow{\sim}x_{22}:=g, x_{12}\overrightarrow{\sim}x_{21}:=i, x_{21}\overrightarrow{\sim}x_{12}:=j\}$, 

\vspace*{0.3cm}

\hspace*{2.7cm} \begin{tikzpicture}\label{Gra:2}
\tikzset{vertex/.style = {shape=circle,draw,minimum size=0.7em}}
\tikzset{edge/.style = {->,> = latex'}}
\node[vertex] (a) at  (0,0) {$u$};
\node[vertex] (b) at  (4,3) {$v$};
\node[vertex] (c) at  (8,0) {$k$};
\node[vertex] (d) at  (4,-3) {$w$};
\draw[edge] (a) to node[above] {e} (b);
\draw[edge] (b) to node[above] {h} (c);
\draw[edge] (a) to node[below] {f} (d);
\draw[edge] (d) to node[below] {g} (c);




\draw[edge] (b) to[bend left] node[right] {i} (d);
\draw[edge] (d) to[bend left] node[left] {j} (b);

\end{tikzpicture}

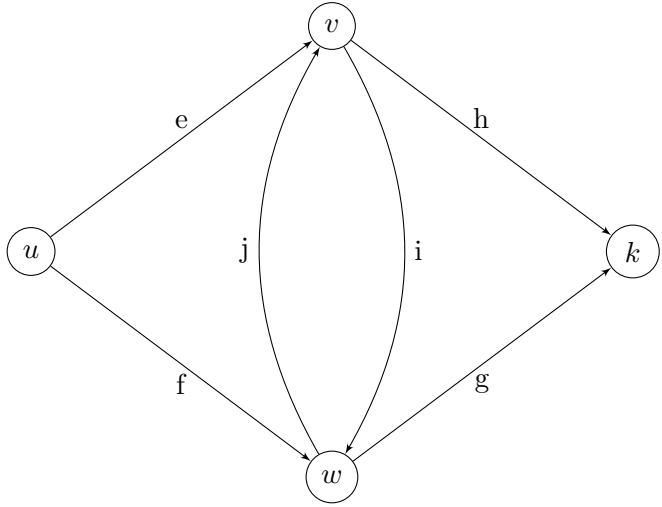
\captionof{figure}{\textbf{Directed locally connected graph related to $\Pi_2$}}

\vspace*{0.3cm}

and we have the following Proposition.

\begin{prop}\label{Prop:CKQ:1}
For $\mathcal{G}(\Pi_2)=(\mathcal{G}^0,\mathcal{G}^1)$ as in \cite[Proposition 9]{RH24}, also Proposition \ref{Prop:mHA:graph} in the current paper, and $\Pi_2$ the associated adjacency matrix. Let $\mathcal{H}:=\ell^2(\mathbb N)$ be the underlying infinite dimensional Hilbert space. Then the set 
\begin{align}
    S=\{& S_e:=\sum_{n=1}^{\infty}E_{6n,3n-2}, S_f:=\sum_{n=1}^{\infty}E_{6n-4,3n-2}, S_h:=\sum_{n=1}^{\infty}E_{6n-3,3n},\notag \\&S_g:=\sum_{n=1}^{\infty}E_{6n-4,3n-1}, S_i:=\sum_{n=1}^{\infty}E_{6n-1,3n}, S_j:=\sum_{n=1}^{\infty}E_{6n-3,3n-1}\}\label{Equ:PIs}
\end{align}
is a Cuntz-Krieger $\mathcal{G}$-family and gives us a graph $C^*$-algebra structure $\mathcal{C}^*(\Pi_2)$. 
\end{prop}
\begin{proof}
    There can be finite or infinite dimension sets of projections. Hence, first, we need to clarify that on which of these spaces we live in!

    For $\{S,P\}\in B(\mathcal{H}):=B(\ell^2(\mathbb N))$, let us start by looking at the constructions on the dimensions of the subspaces $P_v\mathcal{H}, P_w\mathcal{H}, P_u\mathcal{H}$ and $P_k\mathcal{H},$ in order to see that if we need to consider finite-dimensional set of projections or not. We have

     \begin{align}
&\dim(P_k\mathcal{H})=\dim(S_g\mathcal{H})+\dim(S_h\mathcal{H})=\dim(P_v\mathcal{H})+\dim(P_w\mathcal{H})\label{Equ:P1} \\& \dim(P_v\mathcal{H})=\dim(S_e\mathcal{H})+\dim(S_j\mathcal{H})=\dim(P_u\mathcal{H})+\dim(P_w\mathcal{H})\label{Equ:P2}\\&\dim(P_w\mathcal{H})=\dim(S_f\mathcal{H})+\dim(S_i\mathcal{H})=\dim(P_u\mathcal{H})+\dim(P_v\mathcal{H}).\label{Equ:P3}
     \end{align}
From equations \ref{Equ:P1}, \ref{Equ:P2}, and \ref{Equ:P3}, it is not too difficult to conclude that our set of projections can not be finite-dimensional and we should have $\dim(P_v\mathcal{H})=\dim(P_w\mathcal{H})>\infty$, and as an already proved fact, we have that when one of the projections is infinite-dimensional, then they all are forced to be infinite-dimensional. Hence, let us try to construct an appropriate CK-$\Gamma$ family satisfying the equations \ref{Equ:S:1} and \ref{Equ:S:2}. The claim is that the desired set of partial isometries is as introduced in (\ref{Equ:PIs}).

First of all, in order to start proving the above claim, we have to prove the following assertions:
\begin{align*}
   & S_{e}^{*}S_e=P_u, \qquad \qquad \qquad S_{f}^{*}S_f=P_u, \qquad \qquad \qquad S_{h}^{*}S_h=P_v,\\& S_{g}^{*}S_g=P_w, \qquad \qquad \qquad S_{j}^{*}S_j=P_w, \qquad \qquad \qquad S_{i}^{*}S_i=P_v,\\& S_hS_{h}^{*}+S_gS_{g}^{*}=P_k, \qquad S_eS_{e}^{*}+S_jS_{j}^{*}=P_v, \qquad S_fS_{f}^{*}+S_iS_{i}^{*}=P_w.
\end{align*}
By following the above definitions, it is easy to see that we have 
\begin{align*}
    &S_{i}^{*}S_{i}=S_{h}^{*}S_{h}=\sum_{n=1}^{\infty}E_{3n,3n}\\&S_{j}^{*}S_{j}=S_{g}^{*}S_{g}=\sum_{n=1}^{\infty}E_{3n-1,3n-1}\\&S_{e}^{*}S_{e}=S_{f}^{*}S_{f}=\sum_{n=1}^{\infty}E_{3n-2,3n-2},
\end{align*}
and by using these results we see that the statements $S_hS_{h}^{*}+S_gS_{g}^{*}=P_k$, $S_eS_{e}^{*}+S_jS_{j}^{*}=P_v,$ and $S_fS_{f}^{*}+S_iS_{i}^{*}=P_w$ are satisfying. Meaning that $S_e, S_g, S_f, S_h, S_j,$ and $ S_i$ are the desired CK $\Gamma$-family, and the $C^*(S,P)$ is an infinite-dimensional $C^*$-algebra.
\end{proof}

\begin{remark}
\begin{enumerate}
    \item The number of edges in $\mathcal{G}(\Pi_n)$ is equal to $\frac{(n^3+n^2)(n-1)}{2}$.
    \item In what follows, let $\overset{\leftarrow}{\degg}_{hk}$ be the entry degree of the vertex $e_{hk}$. By entry degree we mean the number of edges that are entering to the vertex $e_{hk}$, and analogously we have exit degree of the vertex $e_{hk}$, which will be considered by $\overset{\rightarrow}{\degg}_{hk}$.
\end{enumerate}
\end{remark}

Let us now try to extend the result obtained in Proposition \ref{Prop:CKQ:1} to any $n$ and $\Pi_n$. We have the following claim
\begin{cl}\label{Cl:Op:1}
    For $\Pi_n$ as in Proposition \ref{Prop:mHA:graph} (\cite[Proposition 9]{RH24}), and $\mathcal{G}(\Pi_n)=(\mathcal{G}^0,\mathcal{G}^1)$ the associated locally finite directed graphs, let $\mathcal{H}:=\ell^2(\mathbb N)$ be the underlying infinite dimensional Hilbert space. Then the set 

    $$S=\{ S_i:=\sum_{j=1}^{\infty}\prescript{i}{}{E}_{\mathcal{E}j-A,(n^2-1)j-D} \mid \ \text{for fixed} \ 1\leq i\leq \frac{(n^3+n^2)(n-1)}{2}\}\label{Equ:PIs},$$

is a Cuntz-Krieger $\mathcal{G}$-family for $D\in\{0,\cdots,n^2-2\}$, and $\mathcal{E}$ depends on the degree of the exit edges to the vertex $e_{hk}$, where $i$ is considered as an exit edge, i.e. if $\overset{\rightarrow}{\degg}_{hk}=2$, then we will have $\mathcal{E}=2(n^2-1)$, and if it is 3, then we will have $\mathcal{E}=3(n^2-1)$, and so on, and $A\in\{0,\cdots,\overset{\rightarrow}{\degg}_{hk}\times(n^2-1)\}$, and gives us a graph $C^*$-algebra structure $\mathcal{C}^*(\Pi_n)$.

\end{cl}
We refer the interested reader to For a proof of this claim
Before proving the above claim, let us once again get back to our construction in Proposition \ref{Prop:mHA:graph} and try to present a nondegenerate Cuntz-Krieger algebra structure on the set of graphs $\mathcal{G}(\pi_n)$, for $\pi_n$ the associated commuting matrix with $\Pi_n$. To do this, let us as before, start with the lower dimension algebras, meaning that with $\pi_{2}=\begin{bmatrix}
1 & 0 & 0 & 0\\ 0 & 0 & 1 & 0\\ 0 & 1 & 0 & 0\\ 0 & 0 & 0 & 1
\end{bmatrix}$ and the associated graph $\mathcal{G}(\pi_2)$

\vspace*{0.3cm}

\hspace*{2.1cm} \begin{tikzpicture}
\tikzset{vertex/.style = {shape=circle,draw,minimum size=0.7em}}
\tikzset{edge/.style = {->,> = latex'}}
\node[vertex] (a) at  (0,0) {$v_1$};
\node[vertex] (b) at  (4,3) {$v_2$};
\node[vertex] (c) at  (8,0) {$v_3$};
\node[vertex] (d) at  (4,-3) {$v_4$};
\draw[loop] (a) to node[above] {$e_{11}$} (a);
\draw[loop] (c) to node[above] {$e_{33}$} (c);
\draw[edge] (b) to[bend left] node[right] {$e_{24}$} (d);
\draw[edge] (d) to[bend left] node[left] {$e_{42}$} (b);
\end{tikzpicture}
\vspace*{0.2cm}

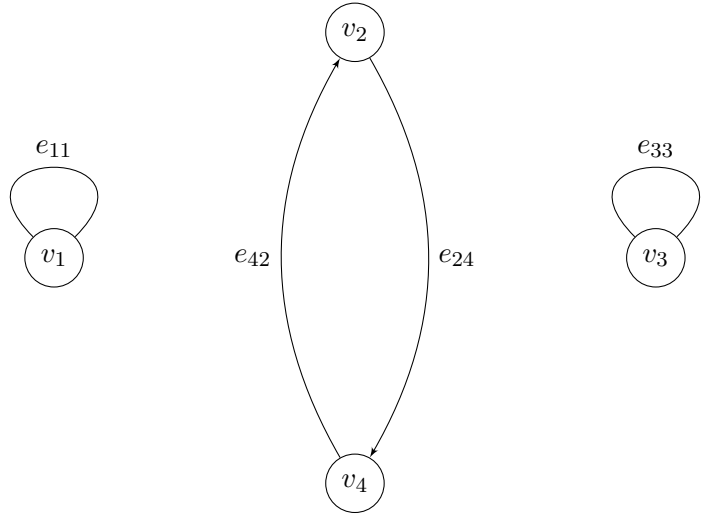
\captionof{figure}{\textbf{Directed 2-connected graph related to $\pi_2$}}

\vspace*{0.7cm}

For graph $\mathcal{G}(\pi_2)$ associated with $\pi_2$, as before consider its set of vertices and edges as $\mathcal{G}^0=\{x_{11}:=v_1, x_{12}:=v_2, x_{22}:=v_3, x_{21}:=v_4\}$ and $\mathcal{G}^1=\{x_{11}\overrightarrow{\sim}x_{11}:=e_{11},  x_{12}\overrightarrow{\sim}x_{21}:=e_{24}, x_{21}\overrightarrow{\sim}x_{12}:=e_{42}, x_{22}\overrightarrow{\sim}x_{22}:=e_{33}\}$, we have the following Proposition.
\begin{prop}\label{Prop:CKQ:2}
For $\mathcal{G}_2:=\mathcal{G}(\pi_2)=(\mathcal{G}^0,\mathcal{G}^1)$ as above, consider $\mathcal{H}$ be the underlying Hilbert space, that can be finite or infinite. Then the set 
\begin{align}
    S=\{& S_{e_{11}}:=E_{2,1}, S_{e_{24}}:=E_{4,1},\notag \\&S_{e_{42}}:=E_{1,4}, S_{e_{33}}:=E_{3,1}\}\label{Equ:PIs}
\end{align}
is a Cuntz-Krieger $\mathcal{G}_2$-family and gives us a graph $C^*$-algebra structure $\mathcal{C}^*(\pi_2):=M_4(\mathbb C)$. 
\end{prop}

\begin{proof}
    There can be finite or infinite dimension sets of projections. Hence, first, we need to clarify that on which of these spaces we may live in!

    For $\{S,P\}\in B(\mathcal{H})$, let us start by looking at the constructions on the dimensions of the subspaces $P_{v_1}\mathcal{H}, P_{v_2}\mathcal{H}, P_{v_3}\mathcal{H}$ and $P_{v_4}\mathcal{H},$ in order to see that if we need to consider finite-dimensional set of projections or not. We have
\begin{align}
&\dim(P_{v_4}\mathcal{H})=\dim(S_{e_{24}}\mathcal{H})=\dim(P_{v_2}\mathcal{H}),\label{Equ:MHG:1}\\&\dim(P_{v_2}\mathcal{H})=\dim(S_{e_{42}}\mathcal{H})=\dim(P_{v_4}\mathcal{H}),\label{Equ:MHG:2}\\&\dim(P_{v_1}\mathcal{H})=\dim(S_{e_{11}}\mathcal{H}),\label{Equ:MHG:3}\\&\dim(P_{v_3}\mathcal{H})=\dim(S_{e_{33}}\mathcal{H}).\label{Equ:MHG:4}
     \end{align}

Now if one of $P_{v_i}\mathcal{H}$, for $i\in\{1,\cdots,4\}$ is infinite-dimensional, then they all are forced to be infinite-dimensional. But, let us try to construct a finite appropriate \textit{CK}-$\mathcal{G}$ family such that $\dim(P_{v_i}\mathcal{H})=1$, and since we require $P_{v_i}$s for $i\in\{1,\cdots,4\}$, to be mutually orthogonal, hence the smallest dimension of $\mathcal{H}$ that we can consider is $1+1+1+1=4$. In this regard, we attempt to construct such a system in $M_4(\mathbb C)$.

Let us take $\{k_1\}$ to be a basis of $P_{v_1}\mathcal{H}$, $\{k_2\}$ to be a basis of $P_{v_2}\mathcal{H}$, $\{k_3\}$ to be a basis of $P_{v_3}\mathcal{H}$, and $\{k_4\}$ to be a basis of $P_{v_4}\mathcal{H}$. Now let $S_{e_{11}}$ sends $k_1$ to $k_1$, $S_{e_{33}}$ sends $k_3$ to $k_3$, $S_{e_{24}}$ sends $k_2$ to $k_4$, and $S_{e_{42}}$ sends $k_4$ to $k_2$, such that we have $S_{e_{11}}=E_{21}$, $S_{e_{33}}=E_{31}$, $S_{e_{24}}=E_{41}$. And the \textit{CK}-conditions tell us that $P_{v_1}=S_{e_{11}}^{*}S_{e_{11}}=E_{12}E_{21}=E_{11}$, $P_{v_3}=S_{e_{33}}^{*}S_{e_{33}}=E_{13}E_{31}=E_{11}$, $P_{v_2}=S_{e_{24}}^{*}S_{e_{24}}=S_{e_{42}}S_{e_{42}}^{*}=E_{14}E_{41}=E_{11}$, and $P_{v_4}=S_{e_{42}}^{*}S_{e_{42}}=S_{e_{24}}S_{e_{24}}^{*}=E_{41}E_{14}=E_{44}$, and so $\{S,P\}$ is indeed a \textit{CK} $\mathcal{G}$-family, and $C^*(S,P)$ is a finite $C^*$-algebra. Since we already have gotten that $E_{11}, E_{12}, E_{13}, E_{14}, E_{21}, E_{22}, E_{31}, E_{33}, E_{41}$, and $E_{44}$ are in $C^*(S,P)$, and the claim is that $C^*(S,P)=M_4(\mathbb C)$, hence in order to prove this claim, what we need is just to show that the rest of matrix units are also belonging to $C^*(S,P)$. We have $E_{23}=E_{21}E_{13}$, meaning that $E_{23}\in C^*(S,P)$, $E_{24}=E_{21}E_{14}\in C^*(S,P)$, $E_{32}=E_{31}E_{12}\in C^*(S,P)$, $E_{34}=E_{31}E_{14}\in C^*(S,P)$, $E_{42}=E_{41}E_{12}\in C^*(S,P)$, and $E_{43}=E_{41}E_{13}\in C^*(S,P)$. Hence $C^*(S,P)$ must be all of $M_4(\mathbb C)$.
\end{proof}

\begin{lemma}
    Let $\mathcal{H}=\ell^2(\mathbb N)$, and $B(\mathcal{H})$ be the respected $C^*$-algebra, and to be our ground $C^*$-algebra. Then it is also possible to construct an infinite-dimensional graph $C^*$-algebra $C^*(\pi_2)$ by considering the following set of \textit{CK} $\mathcal{G}$-family
    \begin{align}
    S=\{&S_{e_{11}}=S_{e_{22}}:=\sum_{n=1}^{\infty}E_{n+1,n}, S_{e_{24}}:=\sum_{n=1}^{\infty}E_{2n,n}, S_{e_{42}}:=\sum_{n=1}^{\infty}E_{n,2n}\}.\label{Equ:PIs}
\end{align}
\end{lemma}
\begin{proof}
    The proof is using the same approach as in the Propositions \ref{Prop:CKQ:1} and \ref{Prop:CKQ:2}.
\end{proof}

Now, as in Proposition \ref{Prop:CKQ:2}, for graph $\mathcal{G}(\pi_n)$ associated with $\pi_n$, consider its set of vertices and edges as $\mathcal{G}^0=\begin{cases}
    x_{ii} \qquad \text{for} \ i\in\{1,\cdots,n\},\\
    x_{ij} \qquad \text{for} \ i,j\in\{1,\cdots,n\} \ \text{and} \ i\neq j.
\end{cases} $ and $\mathcal{G}^1=\begin{cases}
    x_{i,i}\overrightarrow{\sim}x_{i,i}:=e_{i,i},\\
    x_{i,j}\overrightarrow{\sim}x_{j,i}:=e_{2i,2j},
\end{cases}$, we have the following Theorem.

\begin{thm}
We have the following finite and infinite dimensional set of \textit{CK} $\mathcal{G}$-family and the associated $C^*(S,P)$ graph $C^*$-algebras.
\begin{enumerate}
    \item For $\pi_n$ the associated commuting matrix to $\Pi_n$ the adjacency matrix of the graph associated with the defining relations of the coordinate ring of $M_q(n)$, the set
    
    \vspace*{0.3cm}
   \hspace*{3cm} $S=\begin{cases}
        S_{i,i}:=S_{e_{i,i}}:=E_{i+1,1} \qquad \qquad \text{for} \ 1\leq i\leq n\\
        S_{2i,2j}:=S_{e_{2i,2j}}:=E_{2j,1} \qquad \qquad \text{for} \ i\leq j\\
        S_{2i,2j}:=S_{e_{2i,2j}}:=E_{1,2j} \qquad \qquad \text{for} \ j\leq i
    \end{cases}$
    \vspace*{0.3cm}
    
is a Cuntz-Krieger $\mathcal{G}$-family and gives us a graph $C^*$-algebra structure $\mathcal{C}^*(\pi_n)=M_{n^2}(\mathbb C)$.
\item Now let $\mathcal{H}=\ell^2(\mathbb N)$, and $B(\mathcal{H})$ be the respected $C^*$-algebra, and to be our ground $C^*$-algebra. Consider the following set 

\vspace*{0.3cm}
   \hspace*{3cm} $S=\begin{cases}
        S_{i,i}:=S_{e_{ii\overrightarrow{\sim} ii}}:=\sum_{n=1}^{\infty}E_{n+1,n} \qquad \qquad \text{for} \ 1\leq i\leq n\\
        S_{j,2j}:=S_{e_{ij\overrightarrow{\sim} ji}}:=\sum_{n=1}^{\infty}E_{2n,n} \qquad \qquad \text{for} \ i\leq j\\
        S_{2j,j}:=S_{e_{ij\overrightarrow{\sim} ji}}:=\sum_{n=1}^{\infty}E_{n,2n} \qquad \qquad \text{for} \ j\leq i.
    \end{cases}$
    \vspace*{0.3cm}

Then the set $\{S,P\}\in B(\mathcal{H})$ is a \textit{CK} $\mathcal{G}$-family and gives us an infinite dimensional graph $C^*$-algebra structure $C^*(S,P):=C^*(\pi_n)$.
\end{enumerate}

\end{thm}
\begin{proof}
As we work in a nondegenerate space, since the set of vertices of our graphs consists of no sinks, hence there can be finite and infinite dimension sets of projections. Here we will prove both.

Consider $\mathcal{H}$ to be the underlying Hilbert space, and for $\{S,P\}\in B(\mathcal{H})$, let us start by looking at the constructions on the dimensions of the subspaces $P_{v_i}\mathcal{H}$ for $i\in\{1,\cdots,n^2\}$ in order to find out the dimension of the space on which we live in! We have
        \begin{align}
&\dim(P_{x_{ij}}\mathcal{H})=\dim(S_{e_{2i,2j}}\mathcal{H}),\\&\dim(P_{x_{ii}}\mathcal{H})=\dim(S_{e_{i,i}}\mathcal{H}).\label{Equ:MHG:4}
     \end{align}
     Now, as before, if any of $P_{x_{ij}}\mathcal{H}$, for $i,j\in\{1,\cdots,n\}$, are infinite-dimensional, then they all have to be infinite-dimensional.
    \begin{enumerate}
        \item For this case, let us try to construct a finite appropriate \textit{CK} $\mathcal{G}$-family such that $\dim(P_{x_{ij}})=1$, and since we require $P_{x_{ij}}$'s to be mutually orthogonal, hence the smallest dimension of $\mathcal{H}$ that we can consider is $1+\cdots+1=n^2$. So, concerning this, we will try to construct such a system in $M_{n^2}(\mathbb C)$.

        Now take $k_{i,j}$ to be a basis of $P_{x_{ij}}\mathcal{H}$, and let $S_{i,i}$ sends $k_{ii}$ to $k_{ii}$, and $S_{2i,2j}$ sends $k_{ij}$ to $k_{ij}$, for any $i,j\in\{1,\cdots,n\}$ in order to have $S_{i,i}=E_{i+1,1}$, and $S_{2i,2j}=E_{2j,1}$  for $i\leq j$, and $S_{2i,2j}=E_{1,2j}$ for $i\geq j$. Then \text{CK}-conditions tell us that $P_{x_{ii}}=S_{i,i}^{*}S_{i,i}=E_{1,i+1}E_{i+1,1}=E_{1,1}$, and for $i\leq j$ we get $P_{x_{ij}}=S_{2i,2j}^{*}S_{2i,2j}=E_{1,2j}E_{2j,1}=E_{1,1}=S_{2j,2i}S_{2j,2i}^{*}$, and for $i\geq j$ we have $P_{x_{ij}}=S_{2i,2j}^{*}S_{2i,2j}=E_{2j,1}E_{1,2j}=E_{2j,2j}=S_{2j,2i}S_{2j,2i}^{*}$. Hence, $\{S,P\}$ is indeed a \text{CK} $\mathcal{G}$-family, and $C^*(S,P)$ is a finite $C^*$-algebra. Since we already know that $E_{11}, E_{i+1,1}, E_{1,i+1}, E_{2j,1}, E_{1,2j}, E_{2i,2j},$ and $ E_{2j,2i}$ are in $C^*(S,P)$, and the claim is that $C^*(S,P)=M_{n^2}(\mathbb C)$, and in order to prove this claim, we need to show that the rest of matrix units are also belonging to $C^*(S,P)$.

        As we have $E_{i+1,i+1}=E_{i+1,1}E_{1,i+1}\in C^*(S,P)$, $E_{2j,i+1}=E_{2j,1}E_{1,i+1}\in C^*(S,P)$, and $E_{i+1,2j}=E_{i+1,1}E_{1,2j}\in C^*(S,P)$, hence it is almost clear that $C^*(S,P)$ must be all of $M_{n^2}(\mathbb C)$.
         
        \item Let $\mathcal{H}=\ell^2(\mathbb N)$, and $B(\mathcal{H})$ be the respected $C^*$-algebra, and to be our ground $C^*$-algebra. For $\{S,P\}\in B(\mathcal{H})$, in order to prove the claim of the theorem, we are required to prove the following assertions:

        $$\begin{cases}
            S_{i,i}^{*}S_{i,i}=P_{x_{ii}},\\
            S_{2i,2j}^{*}S_{2i,2j}=P_{x_{ij}} \qquad \text{for} \qquad i\leq j,\\
            S_{2i,2j}^{*}S_{2i,2j}=P_{x_{ij}} \qquad \text{for} \qquad i\geq j,
        \end{cases}$$
        then, by following the assumptions of the second part of the theorem, it is not too difficult to see that
        \begin{align*}
            &S_{i,i}^{*}S_{i,i}=\sum_{n=1}^{\infty}E_{n,n} \qquad \qquad \qquad \ \ \ \ \ \ \ \qquad \text{for} \qquad 1\leq i\leq n\\& S_{2i,2j}^{*}S_{2i,2j}=\sum_{n=1}^{\infty}E_{n,n}=S_{2i,2j}S_{2i,2j}^{*} \qquad \ \ \text{for} \qquad  i\leq j\\&S_{2i,2j}^{*}S_{2i,2j}=\sum_{n=1}^{\infty}E_{2n,2n}=S_{2i,2j}S_{2i,2j}^{*} \qquad \text{for} \qquad  i\geq j,
        \end{align*}
    \end{enumerate}
    meaning that $S_{i,i}$ and $S_{2i,2j}$ defined by the relations of the second part of the theorem, for any $i,j\in\{1,\cdots,n\}$, will constitute the desired \textit{CK} $\mathcal{G}$-family, and $C^*(S,P)$ is an infinite-dimensional $C^*$-algebra. 
\end{proof}

\section{Multiplier Hopf $*$-Graph Algebras}
There are many different kinds of graph algebras, but the one in which we are interested should be equipped with a nondegenerate product $\cdot$ meaning that if for any $a$ we have $a\cdot b=0$, then we should get $b=0$ (here $a$ and $b$ are taken from our algebra).

The graph algebra pointed out in the previous section and illustrated in Figure \ref{fig1}, is a nondegenerate $*$-monoid algebra with an identity element considered by the graph associated with the matrix $\pi_{2}=\begin{bmatrix}
1 & 0 & 0 & 0\\ 0 & 0 & 1 & 0\\ 0 & 1 & 0 & 0\\ 0 & 0 & 0 & 1
\end{bmatrix}$, which is the only commuting matrix with the adjacency matrix of the graph associated with $\mathbb K[M_q(2)]$, extendable to $\mathbb K[M_q(n)]$ for any $n\ge 2$, by the rule of assigning 1 to the entries located in the $(ij)(ji)$ position and 0 elsewhere, in the row related to $x_{ij}$, and then we have the following Proposition from \cite[Proposition 9]{RH24}
\begin{prop}\label{Prop:mHA:graph}
	For $\Pi$, a locally finite connected graph associated with coordinate algebra $\mathbb K\left(M_q(n)\right)$ with vertex set $\{x_{11}, x_{12},\cdots, x_{ij}\}$ for $i,j\in \{1,2,\cdots,n\}$ and the index set $I:=\{11, 12, \cdots, ij\}$, there exists a unique universal nondegenerate $*$-algebra $\mathcal{A}$ generated by elements $(u_{hh'})_{h,h'\in I}$, satisfying the relations of quantum permutation groups, and a unique nondegenerate $*$-homomorphism $\Delta:\mathcal{A}\rightarrow M(\mathcal{A}\otimes\mathcal{A})$ satisfying $\Delta(u_{hh'})=\sum_{k\in I}^{}(u_{hk}\otimes u_{kh'})$ for all $h, h'\in I$, such that the pair $(\mathcal{A},\Delta)$ is a multiplier Hopf $*$-algebra in the sense of (\cite{VD94}, Definition 2.4), and since it admits a positive faithful left-invariant (resp. right-invariant) functional, it is an algebraic quantum group in the sense of \cite{AD96}.
\end{prop}
\begin{remark}
    \begin{enumerate}
        \item We note that the defining relations of the graph algebra in Proposition \ref{Prop:mHA:graph}, is not always nondegenerate and we might not be able to use it. But here since the entries $u_{hh'}$ belong to the set $\{0,1\}$ for any $h,h'\in I$, hence we certainly have a nondegenerate binary operator.
        \item Let us discuss a little about different kinds of graph algebras. There are three elementary ways of constructing graph algebras \cite{IR05}. One way is by using the algebra constructed based on the set of vertices of the graph together with the extra vertex 0 (not in the set of vertices). In this way, the algebra will not be nondegenerate and almost useless for our case. The other way is by looking at the entries of the adjacency matrices or the commuting matrices with the adjacency matrices, on which if they are magic unitaries, then by Theorem \ref{Th:RV}, we get a multiplier Hopf algebra structure, but not exactly a multiplier Hopf graph algebra structure!

        The third approach is to directly work on the graph structures and put an algebra structure on the set of specific graphs, as we have done on the set of $\mathcal{G}_i$s in  \cite{RH24}, illustrated in (\ref{GA:1}).
    \end{enumerate}
\end{remark}
Let us call the algebra structure on the vector space of the $(n^2-1)$-connected locally finite graphs $\mathcal{G}_i$, simply $\mathcal{G}$. This is a unital $*$-algebra, and in order to have a $*$-multiplier Hopf algebra, we need to define a map $\Delta$ on $\mathcal{G}$ to $M(\mathcal{G}\otimes\mathcal{G})$, resembling the co-product and satisfying the co-associativity condition
\begin{equation}\label{Equ:Coa}
    (\pi_i\otimes 1\otimes 1)(\Delta\otimes\id)(\Delta(\pi_j)(1\otimes \pi_k))=(\id\otimes\Delta)((\pi_i\otimes 1)\Delta(\pi_j))(1\otimes 1\otimes \pi_k),
\end{equation}
in a way that $\Delta(\pi_i)(1\otimes\pi_j)$ and $(\pi_i\otimes 1)\Delta(\pi_j)$ belong in $\mathcal{G}\otimes\mathcal{G}$, for any $\pi_i, \pi_j, \pi_k\in\mathcal{G}$.

\begin{que}\label{Qu:MGHA:1}
    How can we define such a map for a graph algebra consisting of graphs?
\end{que}
Before answering the above question, we need to come to some conclusions concerning the space of multiplier algebra over $\mathcal{G}$. Here our job is quite easy since our algebra is unital and from \cite[Appendix A]{VD94}, we know that for unital algebras we have $M(\mathcal{G})=\mathcal{G}$ and for the nondegenerate algebras, as in our case, we know that $\mathcal{G}\otimes\mathcal{G}$ will be nondegenerate as well.

\hspace*{0.2cm} Let us recall some facts concerning the multiplier algebra, on which we will need to define the multiplier Hopf graph algebras.

\hspace*{0.2cm} Multiplier algebras are a very important part of studying $C^*$-algebras and their structures. For a unital or non-unital non-degenerate algebra $A$, the associated multiplier algebra $M(A)$ will be defined as a set consisting of maps $\rho:A\to A$ satisfying in $c(\rho(a)b)=(c\rho(a))b$ for all $a,b,c\in A$, meaning that $\rho$ simultaneously left and a right multiplier algebra. $M(A)$ is a unital algebra with unit $1_{M(A)}=(\id_A,\id_A)$, and it is the largest unital algebra that contains $A$ as an essential ideal, and hence $A$ can be naturally embeded in $M(A)$. For more information regarding multiplier Hopf algebras, the interested reader is referred to \cite{RH24}.
\begin{remark}\label{Rem:OD}
    \begin{enumerate}
        \item For $A$ an operator algebra, $A^{**}$ is an operator algebra as well.\label{Rem:OD:1}
        \item For an operator algebra $A$, we have the following completely isometric identification for any $n\geq 1$,
        \begin{equation}
            M_n(A)^{**}\cong M_n(A^{**}).
        \end{equation}\label{Rem:OD:2}
 \end{enumerate}
\end{remark}
As the previous section suggests, the idea is to implement the graph $C^*$-algebra $C^*(S,P)$ associated with $\pi_i$s. In this case, we will obtain $n$ different multiplier $*$-Hopf graph algebras, and in order to define them, let us fix some notation.
let us start with the simplest one as follows:

We note that while working with $\pi_n$, actually our space is a $n^2\times n^2$ block matrix, consisting of $n^n$, $n\times n$ matrices, $n$ blocks in each row and column.  In what comes, our base will be the set $\{1,\cdots,n^2\}$ divided into $n$ partitions $P_1$ to $P_n$, such that each $P_i$, $i\in\{1,\cdots,n\}$, also is divided into $n$ partitions $P_{i}^{j}$ for $j\in\{1,\cdots,n\}$.
\begin{remark}
    Please note that $i$ and $j$ used above, are different than $i$ and $j$ used below.
\end{remark}

\begin{thm}\label{Thm:.:FMHGA}
\begin{enumerate}[label=\Roman*)]
    \item \label{Thm:.:FMHGA:1} For the graph $C^*$-algebra $C^*(S,P):=C^*(\pi_2)=M_4(\mathbb C)$ as in Proposition \ref{Prop:CKQ:2} and the Cuntz-Krieger $\mathcal{G}_2$-family $S=\{S_{e_{1,1}}:=E_{2,1}, S_{e_{2,4}}:=E_{4,1}, S_{e_{42}}:=E_{1,4}, S_{e_{3,3}}:=E_{3,1}\}$, define 
    \begin{align}
        \Delta:&\mathcal{O}(M_4(\mathbb C))[t^{-1}]\to M(\mathcal{O}(M_4(\mathbb C))[t^{-1}]\otimes \mathcal{O}(M_4(\mathbb C))[t^{-1}])\label{Thm:.:FMHGA:1:1}\\&\hspace*{1.46cm} E_{i,j}\hspace*{0.26cm}\longmapsto E_{k,h}\otimes E_{o,r}:=E_{\ell,m},\notag
    \end{align}
    for $\ell=P_{o}^{k}$ and $m= P_{r}^{h}$, expanded linearly on whole of $M_4(\mathbb C)$.  Then $\Delta$ is a coproduct on $\mathcal{O}(M_4(\mathbb C))[t^{-1}]\cong \mathcal{O}(GL(4)) $, and $(\mathcal{O}(GL(4)),\Delta)$ is a multiplier Hopf $*$-graph algebra associated with $\pi_2$, for $i,j,k,h,o,r\in\{1,\cdots,4\}$ and $\ell,m\in\{1,\cdots,8\}$.
    \item As in \cite{VD94}, there exists a unique linear map $\epsilon:\mathcal{O}(GL(4))\to\mathbb C$ taking $E_{i,j}$ to $\delta_{ij}$ such that
    \begin{align}
        &(\epsilon\otimes\id)(\Delta(E_{k,\ell})(1\otimes E_{o,r}))=E_{k,\ell}E_{o,r}\label{Equ:Ep:1}\\&(\id\otimes\epsilon)((E_{k,\ell} \otimes\Delta(E_{o,r}))=E_{k,\ell}E_{o,r},\label{Equ:Ep:2}
    \end{align}
    for all $E_{o,r},E_{k,\ell}$ associated to all $X_{o,r},X_{k,\ell}\in \mathcal{O}(GL(4))$, and $\epsilon$ is a homomorphism. Also there is a unique linear map $S:\mathcal{O}(GL(4))\to M(\mathcal{O}(GL(4)))$ taking $E_{i,j}$ to $E_{j,i}$, associated with $X_{i,j}$ and $X_{j,i}$ respectively, such that
    \begin{align}
        &m(S\otimes\id)(\Delta(E_{o,r})(1\otimes E_{k,\ell}))=\epsilon(E_{o,r}) E_{k,\ell} \label{Equ:An:1}\\&m(\id\otimes S)((E_{o,r} \otimes 1)\Delta(E_{k,\ell}))=\epsilon(E_{k,\ell}) E_{o,r},\label{Equ:An:2}
    \end{align}
    for all $E_{o,r},E_{k,\ell}$ as above, and $m$ denotes multiplication, defined as a linear map from $M(\mathcal{O}(GL(4)))\otimes \mathcal{O}(GL(4))$ to $\mathcal{O}(GL(4))$  and from $\mathcal{O}(GL(4)) \otimes M(\mathcal{O}(GL(4)))$ to $\mathcal{O}(GL(4))$. The map $S$ is an anti-homomorphism.\label{Thm:.:FMHGA:2}
\end{enumerate}
\end{thm}
\begin{proof}
\begin{enumerate}[label=\Roman*)]
    \item First of all please note that we simply can make the definition of $\Delta$ a little bit more rigorous by considering the following:
    \begin{align}
        \Delta:&\mathcal{O}(M_4(\mathbb C))[t^{-1}]\to M(\mathcal{O}(M_4(\mathbb C))[t^{-1}]\otimes \mathcal{O}(M_4(\mathbb C))[t^{-1}])\notag\\&\hspace*{2.46cm}\cong M(\mathcal{O}(M_4(\mathbb C))\otimes \mathcal{O}(M_4(\mathbb C)))[t^{-1}]\notag \\&\hspace*{2.46cm}\cong M(\mathcal{O}(M_4(\mathbb C)\otimes M_4(\mathbb C)))[t^{-1}] \\&\hspace*{2.46cm}\cong \mathcal{O}( M((M_4(\mathbb C)\otimes M_4(\mathbb C))))[t^{-1}] \notag\\&\hspace*{1.46cm} E_{i,j}\hspace*{0.26cm}\longmapsto E_{k,h}\otimes E_{o,r}:=E_{\ell,m},\notag
    \end{align}
    and as $\mathcal{O}(M_4(\mathbb C))[t^{-1}]$ consists of the linear operators $X_{ij}$ taking any matrix to its $ij$ component, with the underlying matrix $E_{ij}$, hence to define $\Delta$, what we need is just to consider it on the underlying matrix $E_{ij}$, and since $M_4(\mathbb C)\otimes M_4(\mathbb C)\cong M_8(\mathbb C)$ as vector spaces, and by the second part \ref{Rem:OD:2} of the Remark \ref{Rem:OD}, it is not too difficult to conclude that $M(M_8(\mathbb C))\cong M_8(M(\mathbb C))=M_8(B(\mathbb C))$, for $B(\mathbb C)$ the space of bounded operators on $\mathbb C$, and hence $\Delta$ has to be a map from $M_4(\mathbb C)$ to $M_8(B(\mathbb C))$, and by using the fact that $E_{k,h}\otimes E_{o,r}$ is in $ M_8(\mathbb C)$ which is a sub-vector space of $M_8(B(\mathbb C))$, we can certify the well-definedness of $\Delta$, for all $E_{k,h}, E_{o,r}\in M_4(\mathbb C)$ and $k,h,o,r\in\{1,\cdots,4\}$, and $\ell,m\in\{1,\cdots,8\}$. In order to prove the co-associativity condition for $\Delta$, we need to prove that condition

\begin{equation}\label{Equ:Coa:2}
    (a\otimes 1\otimes 1)(\Delta\otimes\id)(\Delta(b)(1\otimes c))=(\id\otimes\Delta)((a\otimes 1)\Delta(b))(1\otimes 1\otimes c).
\end{equation}
for any $a,b,c\in \mathcal{O}(M_4(\mathbb C))[t^{-1}]$ satisfies. 

But note that $a,b,c$ will be linear operators $X_{ij}:M_4(\mathbb C)\to \mathbb C$ taking $A$ to its $ij$'s component $a_{ij}$, with underlying matrix $E_{ij}$, and hence, in order to satisfy the condition (\ref{Equ:Coa:2}), we just need to work with the elementary matrices $E_{ij}$ for $i,j\in\{1,\cdots,n\}$. Hence, by considering $1_{M_4(\mathbb C)}=E_{1,1}+E_{2,2}+E_{3,3}+E_{4,4}$, and the definition of $\Delta$, and following the \text{left hand side of (\ref{Equ:Coa:2})}, finally we get (note that the detailed computation can be found in Appendix \ref{Ap:1:})

    \begin{align}
        &\hspace*{-6.76cm}=(E_{k,h}E_{o_{11},1}\otimes E_{1,1}E_{o_{12},1}\otimes E_{1,r_2}E_{y,z}\notag\\&\hspace*{-6.4cm}+E_{k,h}E_{o_{11},1}\otimes E_{2,2}E_{o_{12},1}\otimes E_{1,r_2}E_{y,z}\label{MHGA:1:1}\\&\hspace*{-6.4cm}+E_{k,h}E_{o_{11},1}\otimes E_{3,3}E_{o_{12},1}\otimes E_{1,r_2}E_{y,z}\notag \\&\hspace*{-6.4cm}+E_{k,h}E_{o_{11},1}\otimes E_{4,4}E_{o_{12},1}\otimes E_{1,r_2}E_{y,z},\notag
    \end{align}
and this equation is either 0, or it is equal to one of the summations. Then by setting $o_{12}=1$, and using the Sweedler notation, the equation (\ref{MHGA:1:1}) will be equal to $E_{k,h}E_{o_{1},1}\otimes E_{1,1}\otimes E_{1,r_2}E_{y,z}$.

For the \text{right hand side of (\ref{Equ:Coa:2})}, in a same approach as for the left hand side we get (note that the detailed computation can be found in Appendix \ref{Ap:1:})
    
    \begin{align}
        &\hspace*{-6.4cm}=(E_{k,h}E_{o_1,1}\otimes E_{1,r_{21}}E_{1,1}\otimes E_{1,r_{22}}E_{y,z}\notag\\&\hspace*{-5.9cm}+E_{k,h}E_{o_1,1}\otimes E_{1,r_{21}}E_{2,2}\otimes E_{1,r_{22}}E_{y,z}\label{MHGA:2:2}\\&\hspace*{-5.9cm}+E_{k,h}E_{o_1,1}\otimes E_{1,r_{21}}E_{3,3}\otimes E_{1,r_{22}}E_{y,z}\notag\\&\hspace*{-5.9cm}+E_{k,h}E_{o_1,1}\otimes E_{1,r_{21}}E_{4,4}\otimes E_{1,r_{22}}E_{y,z},\notag
    \end{align}
and exactly in a same way as in the left hand side, (\ref{MHGA:2:2}) is either 0 or it will be equal to one of the summations. So, by setting $r_{21}=1$, and using the Sweedler notation, the equation (\ref{MHGA:2:2}) will be equal to $E_{k,h}E_{o_1,1}\otimes E_{1,1}\otimes E_{1,r_{22}}E_{y,z}$, and by using the Sweedler notation it is easy to see that the two equations (\ref{MHGA:1:1}) and (\ref{MHGA:2:2}), respectively resembling the left and right sides of the co-associativity condition (\ref{Equ:Coa:2}) will be equal, and in order for $\Delta$ to be a comultiplication, we also need $\Delta(a)(1\otimes b)$ and $(a\otimes 1)\Delta(b)$ belong to $M_4(\mathbb C)\otimes M_4(\mathbb C)\cong M_8(\mathbb C)$, which is clearly satisfied. Hence $\Delta$ is a comultiplication.
    \item In order to prove equation (\ref{Equ:Ep:1}), we have
    \begin{align*}
        (\epsilon\otimes\id)(\Delta(E_{k,\ell})(1\otimes E_{o,r}))&=(\epsilon\otimes\id)(E_{k_1,\ell_1}\otimes E_{k_2,\ell_2}E_{o,r})\\&=\epsilon(E_{k_1,\ell_1}) E_{k_2,\ell_2}E_{o,r} \\&=\delta_{k_1,\ell_1} E_{k_2,\ell_2}E_{o,r} \\& \underset{k_1=\ell_1}{=} E_{k_2,\ell_2}E_{o,r},
    \end{align*}
    and by using the Sweedler notation, we have the result. Note that  The second equation (\ref{Equ:Ep:2}) will be proved in the same approach.

    In order to prove equation (\ref{Equ:An:1}), we have 
    \begin{align*}
        m(S\otimes\id)(\Delta(E_{o,r})(1\otimes E_{k,\ell}))&=m(S\otimes\id)(E_{o_1,r_1}\otimes E_{o_2,r_2}E_{k,\ell})\\&=E_{r_1,o_1}(E_{o_2,r_2}E_{k,\ell})\\&=(E_{r_1,o_1}E_{o_2,r_2})E_{k,\ell}\\&\underset{o_1=o_2}{=}E_{r_1,r_2}E_{k,\ell},
    \end{align*}
    and by letting $r=o$ and by using the Sweedler notation we get $r_2=r_1=r=o=o_1=o_2$, such that
    \begin{align*}
       &\hspace*{1.85cm} \underset{r_1=r_2=r}{=}E_{r,r}E_{k,\ell}\\&\hspace*{1.4cm} \underset{E_{r,r}=1_{\mathcal{O}(Gl(4))}}{=} E_{k,\ell},
    \end{align*}
    and in the case of $o_1\neq o_2$ and $r\neq k$, we will get 0, and hence we are done. The proof of the equation (\ref{Equ:An:2}) will proceed in the same approach. Proving the fact that $\epsilon$ is a homomorphism and $S$ an anti-homomorphism, is almost trivial.
    \end{enumerate}
\end{proof}
As before, for graph $\mathcal{G}(\pi_n)$ associated with $\pi_n$, let us consider its set of vertices and edges as $\mathcal{G}^0=\begin{cases}
    x_{ii} \qquad \text{for} \ i\in\{1,\cdots,n\},\\
    x_{ij} \qquad \text{for} \ i,j\in\{1,\cdots,n\} \ \text{and} \ i\neq j.
\end{cases} $ and $\mathcal{G}^1=\begin{cases}
    x_{i,i}\overrightarrow{\sim}x_{i,i}:=e_{i,i},\\
    x_{i,j}\overrightarrow{\sim}x_{j,i}:=e_{2i,2j},
\end{cases}$, we have the following Corollary.
\begin{cor}\label{Cor:>:>:}
\begin{enumerate}[label=\Roman*)]
  \item For the graph $C^*$-algebra $C^*(S,P):=C^*(\pi_n)=M_{n^2}(\mathbb C)$ as in Proposition \ref{Prop:CKQ:2} and the Cuntz-Krieger $\mathcal{G}_n$-family $S=\begin{cases}
       S_{e_{ii}}:=E_{i+1,1}\\
       S_{e_{ij}}:=E_{j,1} \qquad \text{for} \ j\geq i\\
       S_{e_{ij}}:=E_{1,i} \qquad \text{for} \ i\geq j
   \end{cases}$,
   define 
    \begin{align}
        \Delta:&\mathcal{O}(M_{n^2}(\mathbb C))[t^{-1}]\to M(\mathcal{O}(M_{n^2}(\mathbb C))[t^{-1}]\otimes \mathcal{O}(M_{n^2}(\mathbb C))[t^{-1}])\\&\hspace*{1.9cm} E_{i,j}\longmapsto E_{k,h}\otimes E_{o,r}:=E_{\ell,m},
    \end{align}
    for $\ell=P_{o}^{k}$ and $m= P_{r}^{h}$, expanded linearly on whole of $\mathcal{O}(M_{n^2}(\mathbb C))[t^{-1}]$, for $t$ the determinant.  Then $\Delta$ is a coproduct on $\mathcal{O}(M_{n^2}(\mathbb C))[t^{-1}]=\mathcal{O}(GL(n^2))$, and $(\mathcal{O}(GL(n^2)),\Delta)$ is a multiplier Hopf $*$-graph algebra, for $i,j,k,h,o,r\in\{1,\cdots,n^2\}$ and $\ell,m\in\{1,\cdots,2n^2\}$. 
\item  As in \cite{VD94}, and as in Theorem \ref{Thm:.:FMHGA}, there exists a unique linear map $\epsilon:\mathcal{O}(GL(n^2))\to\mathbb C$ taking $E_{i,j}$ to $\delta_{ij}$ such that
    \begin{align}
        &(\epsilon\otimes\id)(\Delta(E_{k,\ell})(1\otimes E_{o,r}))=E_{k,\ell}E_{o,r}\label{Equ:Ep:1}\\&(\id\otimes\epsilon)((E_{k,\ell} \otimes\Delta(E_{o,r}))=E_{k,\ell}E_{o,r},\label{Equ:Ep:2}
    \end{align}
    for all $E_{o,r},E_{k,\ell}$ associated to all $X_{o,r},X_{k,\ell}\in \mathcal{O}(GL(n^2))$, and $\epsilon$ is a homomorphism. Also there is a unique linear map $S:\mathcal{O}(GL(n^2))\to M(\mathcal{O}(GL(n^2)))$ taking $E_{i,j}$ to $E_{j,i}$, associated with $X_{i,j}$ and $X_{j,i}$ respectively, such that
    \begin{align}
        &m(S\otimes\id)(\Delta(E_{o,r})(1\otimes E_{k,\ell}))=\epsilon(E_{o,r}) E_{k,\ell} \label{Equ:An:1}\\&m(\id\otimes S)((E_{o,r} \otimes 1)\Delta(E_{k,\ell}))=\epsilon(E_{k,\ell}) E_{o,r},\label{Equ:An:2}
    \end{align}
    for all $E_{o,r},E_{k,\ell}$ as above, and $m$ denotes multiplication, defined as a linear map from $M(\mathcal{O}(GL(n^2)))\otimes \mathcal{O}(GL(n^2))$ to $\mathcal{O}(GL(n^2))$  and from $\mathcal{O}(GL(n^2)) \otimes M(\mathcal{O}(GL(n^2)))$ to $\mathcal{O}(GL(n^2))$. The map $S$ is an anti-homomorphism.
\end{enumerate}   
\end{cor}
\begin{proof}
    The proof will follow exactly the constructive approaches used in the proof of theorem \ref{Thm:.:FMHGA}.
\end{proof}

\begin{remark}\label{Rem:1:1:1}
    Note that using the $M_{n^2}(\mathbb C)$ matrix spaces is just to emphasis the adaptation from the finite $C^*$-graph algebras associated with the $G(\pi_n)$-Cuntz-Krieger families, and the constructions introduced in Corollary \ref{Cor:>:>:} will also satisfy for the $M_n(\mathbb C)$, and the $GL(n)$ case.
\end{remark}
    
\section{$n$-type discrete quantum groups}\label{Sec::5::}
Moving into the representation theory, the treatment might lay in the Tannakian category of the compact quantum matrix groups (Woronowicz algebras). We have the following well known definition. 
\begin{defn}
    A corepresentation of a compact matrix quantum group (Woronowicz algebra) $A$ is a biunitary matrix $v\in M_n(\mathcal{A})$ satisfying
    \begin{align*}
        \Delta(v_{ij})=\sum\limits_{k}^{}v_{ik}\otimes v_{kj},\qquad \epsilon(v_{ij})=\delta_{ij},\qquad S(v_{ij})=v_{ji}^{*},
    \end{align*}
    where $\mathcal{A}\subset A$ is the dense $*$-subalgebra of ``smooth elements''.
\end{defn}

Now consider $\mathcal{B}\subset M_d(\mathbb C)$, for $\mathcal{B}$ semisimple Lie algebra and $d$ some natural number. We can refer to the above inclusion as $\varphi:\mathcal{B}\to M_d(\mathbb C)$, which is a $d$-dimensional representation of $\mathcal{B}$. It is well known that as an algebra there is an isomorphism $\mathcal{B}=\oplus_{i=1}^{k}\mathcal{B}_i$, and integers $d_i$ such that $\mathcal{B}_i\cong M_{d_i}(\mathbb C)$ for all $i$. Then the tensor product $\mathcal{B}\otimes\mathcal{B}$ of $\mathcal{B}$ with itself can be identified with the direct sum of the algebras $\mathcal{B}_{i_{\alpha}}\otimes\mathcal{B}_{i_{\beta}}$. So we can write $\mathcal{B}\otimes\mathcal{B}=\sum\limits_{\alpha,\beta}^{}\oplus(\mathcal{B}_{\alpha}\otimes\mathcal{B}_{\beta})$. Then the multiplier algebra $M(\mathcal{B}\otimes\mathcal{B})$ can be identified with the product $\Pi_{\alpha,\beta}(\mathcal{B}_{\alpha}\otimes\mathcal{B}_{\beta})$, that is with elements of the form $(x_{\alpha\beta})$ where $x_{\alpha\beta}\in\mathcal{B}_{\alpha}\otimes\mathcal{B}_{\beta}$ for all $\alpha, \beta$, with no further restrictions. Then there is the following definition of the discrete quantum group in the sense of Van Daele. \cite{AD96}
\begin{defn}
    A discrete quantum group is a pair $(A,\Delta)$ where $A$ is a direct sum of full matrix algebras and $\Delta$ is a comultiplication on $A$ making $A$ into a multiplier Hopf $*$-algebra, and the dual of it will be called a compact quantum group.
\end{defn}

Now the idea is to apply these kind of structures to the Corollary \ref{Cor:>:>:}, and obtain some analogous sort of compact and discrete quantum groups for the proposed initial class of examples of the multiplier Hopf $*$-graph algebras, but since $GL(n)$ is not semisimple, hence in order to use Artin-Wedderburn Theorem, we have to employ another objects, and what could be better that $SL(n)$ the special linear group, and we have the following.

We can have an almost same result as in the Corollary \ref{Cor:>:>:} for the $SL(n)$ case as follows.

\begin{thm}\label{Thm:>:>:}
\begin{enumerate}[label=\Roman*)]
  \item For $SL(n,\mathbb C):=SL(n)$, the special linear group, we can define 
    \begin{align}
        \Delta:&\mathcal{O}(M_{n^2}(\mathbb C))[(t-\mathfrak{b})^{-1}]\to M(\mathcal{O}(M_{n^2}(\mathbb C))[(t-\mathfrak{b})^{-1}]\otimes \mathcal{O}(M_{n^2}(\mathbb C))[(t-\mathfrak{b})^{-1}])\\&\hspace*{2.9cm} E_{i,j}\longmapsto E_{k,h}\otimes E_{o,r}:=E_{\ell,m},
    \end{align}
    for $\ell=P_{o}^{k}$ and $m= P_{r}^{h}$, expanded linearly on whole of $\mathcal{O}(M_{n^2}(\mathbb C))[(t-\mathfrak{b})^{-1}]$, for $t$ the determinant, and $\mathfrak{b}\in\mathbb N_0/\{1\}$.  Then $\Delta$ is a coproduct on $\mathcal{O}(M_{n^2}(\mathbb C))[(t-\mathfrak{b})^{-1}]=\mathcal{O}(SL(n^2))$, and $(\mathcal{O}(SL(n^2)),\Delta)$ is a multiplier Hopf $*$-graph algebra, for $i,j,k,h,o,r\in\{1,\cdots,n^2\}$ and $\ell,m\in\{1,\cdots,2n^2\}$. 
\item  As in \cite{VD94}, and as in Theorem \ref{Thm:.:FMHGA}, there exists a unique linear map $\epsilon:\mathcal{O}(SL(n^2))\to\mathbb C$ taking $E_{i,j}$ to $\delta_{ij}$ such that
    \begin{align}
        &(\epsilon\otimes\id)(\Delta(E_{k,\ell})(1\otimes E_{o,r}))=E_{k,\ell}E_{o,r}\label{Equ:Ep:1}\\&(\id\otimes\epsilon)((E_{k,\ell} \otimes\Delta(E_{o,r}))=E_{k,\ell}E_{o,r},\label{Equ:Ep:2}
    \end{align}
    for all $E_{o,r},E_{k,\ell}$ associated to all $X_{o,r},X_{k,\ell}\in \mathcal{O}(SL(n^2))$, and $\epsilon$ is a homomorphism. Also there is a unique linear map $S:\mathcal{O}(SL(n))\to M(\mathcal{O}(SL(n^2)))$ taking $E_{i,j}$ to $E_{j,i}$, associated with $X_{i,j}$ and $X_{j,i}$ respectively, such that
    \begin{align}
        &m(S\otimes\id)(\Delta(E_{o,r})(1\otimes E_{k,\ell}))=\epsilon(E_{o,r}) E_{k,\ell} \label{Equ:An:1}\\&m(\id\otimes S)((E_{o,r} \otimes 1)\Delta(E_{k,\ell}))=\epsilon(E_{k,\ell}) E_{o,r},\label{Equ:An:2}
    \end{align}
    for all $E_{o,r},E_{k,\ell}$ as above, and $m$ denotes multiplication, defined as a linear map from $M(\mathcal{O}(SL(n^2)))\otimes \mathcal{O}(SL(n^2))$ to $\mathcal{O}(SL(n^2))$  and from $\mathcal{O}(SL(n^2)) \otimes M(\mathcal{O}(SL(n^2)))$ to $\mathcal{O}(SL(n^2))$. The map $S$ is an anti-homomorphism.
\end{enumerate}   
\end{thm}
\begin{proof}
    The proof will proceed in a same approach as in the $GL(n^2)$ case.
\end{proof}
\begin{remark}
The same notation issue raised in Remark \ref{Rem:1:1:1}, also will hold here. Meaning that the Theorem \ref{Thm:>:>:} also satisfies for $SL(n)$, and after this point we will work in the space of $n\times n$ matrices instead of the space of $n^2\times n^2$ matrices. 
\end{remark}

And as $SL(n,\mathbb C)$ is semisimple, hence the Artin-Wederburn Theorem will work, and we might have a discrete quantum group structure on $(\mathcal{O}(SL(n)),\Delta)$ in the sense of Van Daele. Let us prove a result almost analogue to a result in \cite[Proposition 3.2]{AD96}.
\begin{prop}\label{Prop:::.:>:2}
    Consider $0\neq F\in\mathcal{O}(SL(n))$ such that $\Delta(G)(1\otimes F)=G\otimes F$, and $\Delta(G)(F\otimes 1)=F\otimes G$, for all $G\in\mathcal{O}(SL(n))$. Then we might assume that $F$ is a projection, and prove that it is unique. Moreover, $GF=FG=\epsilon(G)F$ satisfies for all $G\in\mathcal{O}(SL(n))$.
\end{prop}
\begin{proof}
    The proof is almost trivial, since the objects we are intended to work with are the elementary matrices and the product in this case is well-defined. But, let us have a sketch of the proof. Let $G=E_{ij}$ and $F$ be any element in $\mathcal{O}(SL(n))$. Then we have

    \begin{align}
        E_{ij}\otimes F=\Delta(E_{ij})(1\otimes F)&=(E_{kh}\otimes E_{or})(1\otimes F)\notag\\&=E_{kh}\otimes E_{or}F\label{Equ:p:P1}
    \end{align}
    hence, here we should have $F=E_{rf}$ for some $f\in\{1,\cdots,n\}$, and therefore we have
    \begin{align*}
       \hspace*{2.6cm} \ref{Equ:p:P1}&=E_{kh}\otimes E_{or}E_{rf}\\&=E_{kh}\otimes E_{of}\\&=E_{P_{o}^{k}P_{f}^{h}}.
    \end{align*}
    Therefore, we should have
    $$\hspace*{1.1cm}E_{P_{r}^{i}P_{f}^{j}}=E_{P_{o}^{k}P_{f}^{h}},$$
    meaning that 
    \begin{equation}\label{Equ::l:L:1}
        j=h, i=k, r=o,
    \end{equation}
    
    Following the other equation, we get

\begin{align}
        F\otimes E_{ij}=\Delta(E_{ij})(F\otimes 1)&=(E_{kh}\otimes E_{or})(F\otimes 1)\notag\\&=E_{kh}F\otimes E_{or}\label{Equ:p:P2}
    \end{align}
    hence, here we should have $F=E_{hg}$ for some $g\in\{1,\cdots,n\}$, and therefore we have
    \begin{align*}
       \hspace*{2.6cm} \ref{Equ:p:P2}&=E_{kh}E_{hg}\otimes E_{or}\\&=E_{kg}\otimes E_{or}\\&=E_{P_{o}^{k}P_{r}^{g}}.
    \end{align*}
    Therefore, we should have
    $$\hspace*{0.95cm}E_{P_{i}^{h}P_{j}^{g}}=E_{P_{o}^{k}P_{r}^{g}},$$
    and from (\ref{Equ::l:L:1}) we obtain that
    \begin{equation}\label{Equ::l:L:2}
        k=h, i=o, r=j,
    \end{equation}
    and $f$ has to be equal to $g$, which will provide us with the desired result!
\end{proof}

In order to move forward, consider the linear maps $T_1$ and $T_2$ as in our previous paper \cite{RH24} and in \cite{AD96}, defined on $\mathcal{O}(SL(n))\otimes\mathcal{O}(SL(n))$ by
\begin{align*}
    &T_1(a\otimes b)=\Delta(a)(1\otimes b)\\& T_2(a\otimes b)=(a\otimes 1)\Delta(b),
\end{align*}
for all $a,b\in\mathcal{O}(SL(n))$, and we have the following result, which is almost analogue to the result in \cite[Proposition 3.3]{AD96}.
\begin{prop}\label{Prop:::.:>:1}
    For $F,G\in\mathcal{O}(SL(n))$, if $\Delta(G)(1\otimes F)=0$ implies $F=0$, then the maps $T_1$ and $T_2$ are injective.
\end{prop}
\begin{proof}
Since the proof is not requiring the properties of the specific defined $\Delta$ for $\mathcal{O}(SL(n))$, hence we will just follow the approach used in \cite[Proposition 3.3]{AD96}.
\end{proof}
Then by using the following characterization Theorem, we can conclude one of main results of this paper.
\begin{thm}\cite[Theorem 3.4]{AD96}\label{Thm:::.:>:}
    Let $A$ be such that it satisfies in Artin-Wedderburn decomposition Theorem, and could be written as a direct sum of matrix algebras, and $\Delta$ be a comultiplication on it. Assume $\Delta(A)(A\otimes 1)$ and $\Delta(A)(1\otimes A)$ are equal to $A\otimes A$. Assume there is a non-zero element $h$ such that $\Delta(a)(1\otimes h)=a\otimes h$ for all $a$ and that $\Delta(h)(1\otimes a)=0$ only if $a=0$. Then $(A,\Delta)$ is a discrete quantum group.
\end{thm}
By using Theorem \ref{Thm:::.:>:}, and Propositions \ref{Prop:::.:>:1}, and \ref{Prop:::.:>:2}, and the defining relations of the multiplier Hopf $*$-graph algebra $(\mathcal{O}(SL(n)),\Delta)$ constructed in Theorem \ref{Thm:>:>:}, we have the following result.
\begin{cor}\label{COr:2:2:3:4}
    $(\mathcal{O}(SL(n)),\Delta)$, constructed in Theorem \ref{Thm:>:>:}, possesses the structures of being a discrete quantum group, and will be called the type $n$-special discrete quantum group graph.
\end{cor}

Now, there arises a natural question raised by Wang \cite{SW99}, related with the $q$-deformation of the finite groups of Lie type, and somehow generalizing the $q$-deformation of the classical Lie groups.

\begin{que}\label{OP:<:<:}
    Do finite groups of Lie type have an analogue of q-deformations into finite quantum groups?
\end{que}

In \cite{SW99}, a deformation of Rieffel type for finite groups (actually finite quantum groups) that contain an Abelian subgroup is constructed. However this deformation is not the analogue of the $q$-deformations, it is rather an analogue of the Drinfel'd twisting.

On the other hand, we hope that structures introduced in this paper will finally help us and will provide us with some new directions in order to consider the above mentioned open problem \ref{OP:<:<:}, specially if we could generalize our results for the case of $SU(n,\mathbb C), SO(n,\mathbb C), SP(n,\mathbb C)$ and the whole category of the finite groups of Lie type!

\vspace*{0.2cm}
To do this, let us start from $O(n,\mathbb C):=M_n(\mathbb C)/(A^TA-I)$, and $U(n,\mathbb C):=M_n(\mathbb C)/(AA^{*}-I)$ and define $ \mathcal{O}(SO(n)):=\mathcal{O}(SO(n,\mathbb C))=\mathcal{O}(O(n,\mathbb C))[(t-\mathfrak{b})^{-1}$, and $ \mathcal{O}(SU(n)):=\mathcal{O}(SU(n,\mathbb C))=\mathcal{O}(U(n,\mathbb C))[(t-\mathfrak{b})^{-1}$ for $t$ the determinant and $\mathfrak{b}\in\mathbb N_0/\{1\}$. Then we may have a result in a same approach as in the Corollary \ref{COr:2:2:3:4} for $\mathcal{O}(SO(n))$ and $\mathcal{O}(SU(n))$ as follows:
\begin{cor}
    $(\mathcal{O}(SO(n)),\Delta)$, and $(\mathcal{O}(SU(n)),\Delta)$ for $\Delta$ constructed as in Theorem \ref{Thm:>:>:}, possess the structures of being a discrete quantum group, and will be called the type $n$-orthogonal and the type $n$-unitary discrete quantum group graphs, respectively.
\end{cor}
\begin{proof}
    The proof will be proceed as in the Corollary \ref{COr:2:2:3:4}.
\end{proof}

\section{Concluding remarks}
We believe that the research conducted in this paper is very interesting, and if we want to describe it in just one sentence it would be ``from simplicity to complexity''! 

\hspace*{0.2cm}Once again, and as in our previous work \cite{RH24}, we started by initializing our toy example $\mathbb K[M_q(n)]$, and by getting motivated through the graph $C^*$-algebras and the Cuntz-Krieger $\Gamma$-family for the finite connected directed graph $\Gamma$, we succeeded in introducing the first class of initial examples of the multiplier Hopf $*$-graph algebras, by looking at the finite $C^*$-graph algebras associated with the commuting matrices $\pi_i$ of the adjacency matrices $\Pi_i$ associated to the coordinate ring of $\mathbb K[M_q(i)]$, for $i\in\{2,\cdots,n\}$. 

\hspace*{0.2cm} Despite the sharpness and the constructive tools used in introducing the results obtained in this paper, we have to admit that there still remains some open problems, such as the Claim proposed in \ref{Cl:Op:1}, and the raised question of the existing multiplier Hopf $*$-graph algebras based on just the graph algebras coming from the set of certain finite locally connected (directed) graphs, as has been illustrated in the Question \ref{Qu:MGHA:1}. Also to try to do the same approaches for another set of regulated (directed) graphs, might provide us with some other interesting results!

By moving to the representation theory, in Section \ref{Sec::5::}, we introduced the $n$-type (special $\&$ unitary) discrete quantum groups based on the structures introduced and studied by Van Daele. We proved that $\mathcal{O}(SL(n))$ and $\mathcal{O}(SU(n))$ equipped with the comultiplication from the Corollary \ref{Cor:>:>:}, making $\mathcal{O}(GL(n))$ into a multiplier Hopf $*$-graph algebra, possesses the structures of being discrete quantum groups, and paving the way to move even further in proving an almost same assertions for the other finite groups of Lie type, and making one step forward in attacking the famous question \ref{OP:<:<:} raised by Wang \cite{SW99}, asking if there are $q$-deformations of finite groups of Lie type in the sense of finite quantum groups, which still is unknown!

For the future works, we have to point it out that the constructions made in this paper are in direct relation with the theory of quantum computing and the quantum information theory through the Cuntz-Krieger $\Gamma$-families $\{S,P\}$, and the quantum graph theory, for some certain locally connected finite directed graphs $\Gamma$. Our next paper will concern about these (pre-)theories, on which we will try to propose a new method of studying the (entangled/not-entangled) $q$-qubit quantum systems!

So, we strongly encourage the interested reader to stay tuned, and we are very pleased to welcome any comments concerning the constructions made in this paper, and in our previous paper \cite{RH24}, and hopefully the next paper \cite{RH243}!

\section{Acknowledgement}
\hspace*{0.2cm}  
F.R. is supported by the Azarbaijan Shahid Madani University under the grant contract No. 117.d.22844 - 08.07.2023. F.R. is also partially supported by a grant from the Institute for Research in Fundamental Sciences (IPM), with the grant (No. 1403170014). Also F.R. would like to express his warmest thanks and gratitude to the late Professor N. A. Vavilov, whom he sees himself indebted for the continuation of his scientific life, as being a member of his 2021 \textit{Ph.D.} defense committee!

Apart from the acknowledgments related with the financial supports, F.R. also would like to express his sincere thanks and gratitude for the hospitality of the Department of Mathematics of the Azarbaijan Shahid Madani University, where all of the work conducted in this paper has been extracted and has been finalized!



\def\cprime{$'$} \def\cprime{$'$} \def\cprime{$'$}
\providecommand{\bysame}{\leavevmode\hbox to3em{\hrulefill}\thinspace}
\providecommand{\MR}{\relax\ifhmode\unskip\space\fi MR }
\providecommand{\MRhref}[2]{%
	\href{http://www.ams.org/mathscinet-getitem?mr=#1}{#2}
}
\providecommand{\href}[2]{#2}

\let\itshape\upshape








\section{Appendix}\label{Ap:1:}
This appendix is concerned with the complete proof of the co-associativity of $\Delta$, defined in (\ref{Thm:.:FMHGA:1:1}) in Theorem \ref{Thm:.:FMHGA}.\ref{Thm:.:FMHGA:1}.

Following the remaining of the proof of the Theorem, we have to prove that for all $a,b,c\in\mathcal{O}(M_4(\mathbb C))[t^{-1}]$, we have that $\Delta$ satisfies in the following co-associativity condition

\begin{equation}\label{Equ:Coa:1}
    (a\otimes 1\otimes 1)(\Delta\otimes\id)(\Delta(b)(1\otimes c))=(\id\otimes\Delta)((a\otimes 1)\Delta(b))(1\otimes 1\otimes c).
\end{equation}
Note that as before $a,b,c$ will be linear operators $X_{ij}:M_4(\mathbb C)\to \mathbb C$ taking $A$ to its $ij$'s component $a_{ij}$, with underlying matrix $E_{ij}$, and hence, in order to satisfy the condition (\ref{Equ:Coa:1}), we just need to work with the elementary matrices $E_{ij}$ for $i,j\in\{1,\cdots,n\}$. 

Consider $1_{M_4(\mathbb C)}=E_{1,1}+E_{2,2}+E_{3,3}+E_{4,4}$, and let us try to prove the co-associativity condition. For the \text{left hand side of (\ref{Equ:Coa})} we have
    
    \begin{align*}
      &\hspace*{-2.3cm}  =(E_{k,h}\otimes 1_{M_4(\mathbb C)}\otimes 1_{M_4(\mathbb C)})(\Delta\otimes \id)(\Delta(E_{o,r})(1_{M_4(\mathbb C)}\otimes E_{y,z}))\\&\hspace*{-2.3cm}=(E_{k,h}\otimes 1_{M_4(\mathbb C)}\otimes 1_{M_4(\mathbb C)})(\Delta\otimes \id)(\Delta(E_{o,r})((E_{1,1}+E_{2,2}+E_{3,3}+E_{4,4})\otimes E_{y,z}))\\&\hspace*{-2.3cm}=(E_{k,h}\otimes 1_{M_4(\mathbb C)}\otimes 1_{M_4(\mathbb C)})(\Delta\otimes \id)(\Delta(E_{o,r})\\&\hspace*{-1.76cm}(E_{1,1}\otimes E_{y,z}+E_{2,2}\otimes E_{y,z}+E_{3,3}\otimes E_{y,z}+E_{4,4}\otimes E_{y,z}))\\&\hspace*{-2.3cm}=(E_{k,h}\otimes 1_{M_4(\mathbb C)}\otimes 1_{M_4(\mathbb C)})(\Delta\otimes \id)((E_{o_1,r_1}\otimes E_{o_2,r_2})\\&\hspace*{-1.76cm}(E_{1,1}\otimes E_{y,z}+E_{2,2}\otimes E_{y,z}+E_{3,3}\otimes E_{y,z}+E_{4,4}\otimes E_{y,z}))\\&\hspace*{-2.3cm}=(E_{k,h}\otimes 1_{M_4(\mathbb C)}\otimes 1_{M_4(\mathbb C)})(\Delta\otimes \id)\\&\hspace*{-1.76cm}((E_{o_1,r_1}E_{1,1}\otimes E_{o_2,r_2}E_{y,z}+E_{o_1,r_1}E_{2,2}\otimes E_{o_2,r_2}E_{y,z}\\&\hspace*{-1.9cm}+E_{o_1,r_1}E_{3,3}\otimes E_{o_2,r_2}E_{y,z}+E_{o_1,r_1}E_{4,4}\otimes E_{o_2,r_2}E_{y,z}))
    \end{align*}
    take $r_1=1$, then we have
    \begin{align*}
        &\hspace*{-3.15cm}=(E_{k,h}\otimes 1_{M_4(\mathbb C)}\otimes 1_{M_4(\mathbb C)})(\Delta\otimes \id)(E_{o_1,1}\otimes E_{o_2,r_2}E_{y,z})\\&\hspace*{-3.15cm}=(E_{k,h}\otimes (E_{1,1}+E_{2,2}+E_{3,3}+E_{4,4})\otimes (E_{1,1}+E_{2,2}+E_{3,3}+E_{4,4}))\\&\hspace*{-2.65cm}(\Delta\otimes \id)(E_{o_1,1}\otimes E_{o_2,r_2}E_{y,z})\\&\hspace*{-3.15cm}=(E_{k,h}\otimes E_{1,1}\otimes E_{1,1}+ E_{k,h}\otimes E_{1,1}\otimes E_{2,2}+ E_{k,h}\otimes E_{1,1}\otimes E_{3,3}\\&\hspace*{-2.7cm}+ E_{k,h}\otimes E_{1,1}\otimes E_{4,4}+ E_{k,h}\otimes E_{2,2}\otimes E_{1,1}+ E_{k,h}\otimes E_{2,2}\otimes E_{2,2}\\&\hspace*{-2.7cm}+ E_{k,h}\otimes E_{2,2}\otimes E_{3,3}+ E_{k,h}\otimes E_{2,2}\otimes E_{4,4}+ E_{k,h}\otimes E_{3,3}\otimes E_{1,1} \\&\hspace*{-2.7cm}+E_{k,h}\otimes E_{3,3}\otimes E_{2,2} +E_{k,h}\otimes E_{3,3}\otimes E_{3,3} +E_{k,h}\otimes E_{3,3}\otimes E_{4,4} \\&\hspace*{-2.7cm}+E_{k,h}\otimes E_{4,4}\otimes E_{1,1}+E_{k,h}\otimes E_{4,4}\otimes E_{2,2}+E_{k,h}\otimes E_{4,4}\otimes E_{3,3}\\&\hspace*{-2.7cm}+E_{k,h}\otimes E_{4,4}\otimes E_{4,4})(E_{o_{11},1}\otimes E_{o_{12},1}\otimes E_{o_2,r_2}E_{y,z})
        \\&\hspace*{-3.15cm}=(E_{k,h}E_{o_{11},1}\otimes E_{1,1}E_{o_{12},1}\otimes E_{1,1}E_{o_2,r_2}E_{y,z}\\&\hspace*{-2.7cm}+E_{k,h}E_{o_{11},1}\otimes E_{1,1}E_{o_{12},1}\otimes E_{2,2}E_{o_2,r_2}E_{y,z}\\&\hspace*{-2.7cm}+E_{k,h}E_{o_{11},1}\otimes E_{1,1}E_{o_{12},1}\otimes E_{3,3}E_{o_2,r_2}E_{y,z}\\&\hspace*{-2.7cm}+E_{k,h}E_{o_{11},1}\otimes E_{1,1}E_{o_{12},1}\otimes E_{4,4}E_{o_2,r_2}E_{y,z}\\&\hspace*{-2.7cm}+E_{k,h}E_{o_{11},1}\otimes E_{2,2}E_{o_{12},1}\otimes E_{1,1}E_{o_2,r_2}E_{y,z}
        \end{align*}

        \begin{align*}
        &\hspace*{-5.3cm}+E_{k,h}E_{o_{11},1}\otimes E_{2,2}E_{o_{12},1}\otimes E_{2,2}E_{o_2,r_2}E_{y,z}\\&\hspace*{-5.3cm}+E_{k,h}E_{o_{11},1}\otimes E_{2,2}E_{o_{12},1}\otimes E_{3,3}E_{o_2,r_2}E_{y,z}\\&\hspace*{-5.3cm}+ E_{k,h}E_{o_{11},1}\otimes E_{2,2}E_{o_{12},1}\otimes E_{4,4}E_{o_2,r_2}E_{y,z}\\&\hspace*{-5.3cm}+ E_{k,h}E_{o_{11},1}\otimes E_{3,3}E_{o_{12},1}\otimes E_{1,1}E_{o_2,r_2}E_{y,z} \\&\hspace*{-5.3cm}+E_{k,h}E_{o_{11},1}\otimes E_{3,3}E_{o_{12},1}\otimes E_{2,2}E_{o_2,r_2}E_{y,z} \\&\hspace*{-5.3cm}+E_{k,h}E_{o_{11},1}\otimes E_{3,3}E_{o_{12},1}\otimes E_{3,3}E_{o_2,r_2}E_{y,z} \\&\hspace*{-5.3cm}+E_{k,h}E_{o_{11},1}\otimes E_{3,3}E_{o_{12},1}\otimes E_{4,4}E_{o_2,r_2}E_{y,z} \\&\hspace*{-5.3cm}+E_{k,h}E_{o_{11},1}\otimes E_{4,4}E_{o_{12},1}\otimes E_{1,1}E_{o_2,r_2}E_{y,z}\\&\hspace*{-5.3cm}+E_{k,h}E_{o_{11},1}\otimes E_{4,4}E_{o_{12},1}\otimes E_{2,2}E_{o_2,r_2}E_{y,z}\\&\hspace*{-5.3cm}+E_{k,h}E_{o_{11},1}\otimes E_{4,4}E_{o_{12},1}\otimes E_{3,3}E_{o_2,r_2}E_{y,z}\\&\hspace*{-5.3cm}+E_{k,h}E_{o_{11},1}\otimes E_{4,4}E_{o_{12},1}\otimes E_{4,4}E_{o_2,r_2}E_{y,z}),
    \end{align*}
    and by the next consideration, and by letting $o_2=1$, we have

    \begin{align*}
        &\hspace*{-6.15cm}=(E_{k,h}E_{o_{11},1}\otimes E_{1,1}E_{o_{12},1}\otimes E_{1,1}E_{1,r_2}E_{y,z}\\&\hspace*{-5.7cm}+E_{k,h}E_{o_{11},1}\otimes E_{1,1}E_{o_{12},1}\otimes E_{2,2}E_{1,r_2}E_{y,z}\\&\hspace*{-5.7cm}+E_{k,h}E_{o_{11},1}\otimes E_{1,1}E_{o_{12},1}\otimes E_{3,3}E_{1,r_2}E_{y,z}\\&\hspace*{-5.7cm}+E_{k,h}E_{o_{11},1}\otimes E_{1,1}E_{o_{12},1}\otimes E_{4,4}E_{1,r_2}E_{y,z}\\&\hspace*{-5.7cm}+E_{k,h}E_{o_{11},1}\otimes E_{2,2}E_{o_{12},1}\otimes E_{1,1}E_{1,r_2}E_{y,z}
        \end{align*}
        \begin{align*}
        &\hspace*{-5.7cm}+E_{k,h}E_{o_{11},1}\otimes E_{2,2}E_{o_{12},1}\otimes E_{2,2}E_{1,r_2}E_{y,z}\\&\hspace*{-5.7cm}+E_{k,h}E_{o_{11},1}\otimes E_{2,2}E_{o_{12},1}\otimes E_{3,3}E_{1,r_2}E_{y,z}\\&\hspace*{-5.7cm}+ E_{k,h}E_{o_{11},1}\otimes E_{2,2}E_{o_{12},1}\otimes E_{4,4}E_{1,r_2}E_{y,z}\\&\hspace*{-5.7cm}+E_{k,h}E_{o_{11},1}\otimes E_{3,3}E_{o_{12},1}\otimes E_{1,1}E_{1,r_2}E_{y,z} \\&\hspace*{-5.7cm}+E_{k,h}E_{o_{11},1}\otimes E_{3,3}E_{o_{12},1}\otimes E_{2,2}E_{1,r_2}E_{y,z} \\&\hspace*{-5.7cm}+E_{k,h}E_{o_{11},1}\otimes E_{3,3}E_{o_{12},1}\otimes E_{3,3}E_{1,r_2}E_{y,z} \\&\hspace*{-5.7cm}+E_{k,h}E_{o_{11},1}\otimes E_{3,3}E_{o_{12},1}\otimes E_{4,4}E_{1,r_2}E_{y,z} \\&\hspace*{-5.7cm}+E_{k,h}E_{o_{11},1}\otimes E_{4,4}E_{o_{12},1}\otimes E_{1,1}E_{1,r_2}E_{y,z}\\&\hspace*{-5.7cm}+E_{k,h}E_{o_{11},1}\otimes E_{4,4}E_{o_{12},1}\otimes E_{2,2}E_{1,r_2}E_{y,z}\\&\hspace*{-5.7cm}+E_{k,h}E_{o_{11},1}\otimes E_{4,4}E_{o_{12},1}\otimes E_{3,3}E_{1,r_2}E_{y,z}\\&\hspace*{-5.7cm}+E_{k,h}E_{o_{11},1}\otimes E_{4,4}E_{o_{12},1}\otimes E_{4,4}E_{1,r_2}E_{y,z}),
    \end{align*}
    resulting the following

    \begin{align}
        &\hspace*{-6.76cm}=(E_{k,h}E_{o_{11},1}\otimes E_{1,1}E_{o_{12},1}\otimes E_{1,r_2}E_{y,z}\notag\\&\hspace*{-6.4cm}+E_{k,h}E_{o_{11},1}\otimes E_{2,2}E_{o_{12},1}\otimes E_{1,r_2}E_{y,z}\label{MHGA:1}\\&\hspace*{-6.4cm}+E_{k,h}E_{o_{11},1}\otimes E_{3,3}E_{o_{12},1}\otimes E_{1,r_2}E_{y,z}\notag \\&\hspace*{-6.4cm}+E_{k,h}E_{o_{11},1}\otimes E_{4,4}E_{o_{12},1}\otimes E_{1,r_2}E_{y,z},\notag
    \end{align}
    and this equation is either 0, or it is equal to one of the summations. Then by setting $o_{12}=1$, and using the Sweedler notation, the equation (\ref{MHGA:1}) will be equal to $E_{k,h}E_{o_{1},1}\otimes E_{1,1}\otimes E_{1,r_2}E_{y,z}$.

    For the \text{right hand side of (\ref{Equ:Coa})} we have
    \begin{align*}
       &\hspace*{-2.3cm} =(\id\otimes\Delta)((E_{k,h}\otimes 1_{M_4(\mathbb C)})\Delta(E_{o,r}))(1_{M_4(\mathbb C)}\otimes1_{M_4(\mathbb C)}\otimes E_{y,z})\\&\hspace*{-2.3cm} =(\id\otimes\Delta)((E_{k,h}\otimes 1_{M_4(\mathbb C)})\Delta(E_{o,r}))\\&\hspace*{-1.73cm} ((E_{1,1}+E_{2,2}+E_{3,3}+E_{4,4})\otimes(E_{1,1}+E_{2,2}+E_{3,3}+E_{4,4})\otimes E_{y,z})
     \end{align*}
     \begin{align*}
     &\hspace*{-2.3cm}=(\id\otimes\Delta)((E_{k,h}\otimes (E_{1,1}+E_{2,2}+E_{3,3}+E_{4,4}))(E_{o_1,r_1}\otimes E_{o_2,r_2}))\\&\hspace*{-1.73cm} (E_{1,1}\otimes E_{1,1}\otimes E_{y,z}+E_{1,1}\otimes E_{2,2}\otimes E_{y,z}\\&\hspace*{-1.73cm}+E_{1,1}\otimes E_{3,3}\otimes E_{y,z}+E_{1,1}\otimes E_{4,4}\otimes E_{y,z}\\&\hspace*{-1.73cm}+E_{2,2}\otimes E_{1,1}\otimes E_{y,z} +E_{2,2}\otimes E_{2,2}\otimes E_{y,z}\\&\hspace*{-1.73cm}+E_{2,2}\otimes E_{3,3}\otimes E_{y,z}+E_{2,2}\otimes E_{4,4}\otimes E_{y,z}\\&\hspace*{-1.73cm}+E_{3,3}\otimes E_{1,1}\otimes E_{y,z}+E_{3,3}\otimes E_{2,2}\otimes E_{y,z}\\&\hspace*{-1.73cm}+E_{3,3}\otimes E_{3,3}\otimes E_{y,z}+E_{3,3}\otimes E_{4,4}\otimes E_{y,z}\\&\hspace*{-1.73cm}+E_{4,4}\otimes E_{1,1}\otimes E_{y,z}+E_{4,4}\otimes E_{2,2}\otimes E_{y,z}\\&\hspace*{-1.73cm}+E_{4,4}\otimes E_{3,3}\otimes E_{y,z}+E_{4,4}\otimes E_{4,4}\otimes E_{y,z})\\&\hspace*{-2.3cm} =(\id\otimes\Delta)((E_{k,h}\otimes E_{1,1}+E_{k,h}\otimes E_{2,2}+E_{k,h}\otimes E_{3,3}+E_{k,h}\otimes E_{4,4})\\&\hspace*{-1.73cm} (E_{o_1,r_1}\otimes E_{o_2,r_2}))(E_{1,1}\otimes E_{1,1}\otimes E_{y,z}+E_{1,1}\otimes E_{2,2}\otimes E_{y,z}\\&\hspace*{-1.73cm}+E_{1,1}\otimes E_{3,3}\otimes E_{y,z}+E_{1,1}\otimes E_{4,4}\otimes E_{y,z}\\&\hspace*{-1.73cm}+E_{2,2}\otimes E_{1,1}\otimes E_{y,z} +E_{2,2}\otimes E_{2,2}\otimes E_{y,z}\\&\hspace*{-1.73cm}+E_{2,2}\otimes E_{3,3}\otimes E_{y,z}+E_{2,2}\otimes E_{4,4}\otimes E_{y,z}\\&\hspace*{-1.73cm}+E_{3,3}\otimes E_{1,1}\otimes E_{y,z}+E_{3,3}\otimes E_{2,2}\otimes E_{y,z}\\&\hspace*{-1.73cm}+E_{3,3}\otimes E_{3,3}\otimes E_{y,z}+E_{3,3}\otimes E_{4,4}\otimes E_{y,z}\\&\hspace*{-1.73cm}+E_{4,4}\otimes E_{1,1}\otimes E_{y,z}+E_{4,4}\otimes E_{2,2}\otimes E_{y,z}\\&\hspace*{-1.73cm}+E_{4,4}\otimes E_{3,3}\otimes E_{y,z}+E_{4,4}\otimes E_{4,4}\otimes E_{y,z})\\&\hspace*{-2.3cm} =(\id\otimes\Delta)(E_{k,h}E_{o_1,r_1}\otimes E_{1,1}E_{o_2,r_2}+E_{k,h}E_{o_1,r_1}\otimes E_{2,2}E_{o_2,r_2}\\&\hspace*{-1.73cm}+E_{k,h}E_{o_1,r_1}\otimes E_{3,3}E_{o_2,r_2}+E_{k,h}E_{o_1,r_1}\otimes E_{4,4}E_{o_2,r_2})\\&\hspace*{-1.73cm} (E_{1,1}\otimes E_{1,1}\otimes E_{y,z}+E_{1,1}\otimes E_{2,2}\otimes E_{y,z}\\&\hspace*{-1.73cm}+E_{1,1}\otimes E_{3,3}\otimes E_{y,z}+E_{1,1}\otimes E_{4,4}\otimes E_{y,z}\\&\hspace*{-1.73cm}+E_{2,2}\otimes E_{1,1}\otimes E_{y,z} +E_{2,2}\otimes E_{2,2}\otimes E_{y,z}\\&\hspace*{-1.73cm}+E_{2,2}\otimes E_{3,3}\otimes E_{y,z}+E_{2,2}\otimes E_{4,4}\otimes E_{y,z}\\&\hspace*{-1.73cm}+E_{3,3}\otimes E_{1,1}\otimes E_{y,z}+E_{3,3}\otimes E_{2,2}\otimes E_{y,z}\\&\hspace*{-1.73cm}+E_{3,3}\otimes E_{3,3}\otimes E_{y,z}+E_{3,3}\otimes E_{4,4}\otimes E_{y,z}\\&\hspace*{-1.73cm}+E_{4,4}\otimes E_{1,1}\otimes E_{y,z}+E_{4,4}\otimes E_{2,2}\otimes E_{y,z}\\&\hspace*{-1.73cm}+E_{4,4}\otimes E_{3,3}\otimes E_{y,z}+E_{4,4}\otimes E_{4,4}\otimes E_{y,z}),
    \end{align*}

    take $o_2=1$, then we have
    \begin{align*}
        &\hspace*{-4.3cm} =(\id\otimes\Delta)(E_{k,h}E_{o_1,r_1}\otimes E_{1,r_2})\\&\hspace*{-3.75cm} (E_{1,1}\otimes E_{1,1}\otimes E_{y,z}+E_{1,1}\otimes E_{2,2}\otimes E_{y,z}\\&\hspace*{-3.75cm}+E_{1,1}\otimes E_{3,3}\otimes E_{y,z}+E_{1,1}\otimes E_{4,4}\otimes E_{y,z}\\&\hspace*{-3.75cm}+E_{2,2}\otimes E_{1,1}\otimes E_{y,z} +E_{2,2}\otimes E_{2,2}\otimes E_{y,z}\\&\hspace*{-3.75cm}+E_{2,2}\otimes E_{3,3}\otimes E_{y,z}+E_{2,2}\otimes E_{4,4}\otimes E_{y,z}\\&\hspace*{-3.75cm}+E_{3,3}\otimes E_{1,1}\otimes E_{y,z}+E_{3,3}\otimes E_{2,2}\otimes E_{y,z}
    \end{align*}
    \begin{align*}
    &\hspace*{-3.75cm}+E_{3,3}\otimes E_{3,3}\otimes E_{y,z}+E_{3,3}\otimes E_{4,4}\otimes E_{y,z}\\&\hspace*{-3.75cm}+E_{4,4}\otimes E_{1,1}\otimes E_{y,z}+E_{4,4}\otimes E_{2,2}\otimes E_{y,z}\\&\hspace*{-3.75cm}+E_{4,4}\otimes E_{3,3}\otimes E_{y,z}+E_{4,4}\otimes E_{4,4}\otimes E_{y,z})\\&\hspace*{-4.3cm}=(E_{k,h}E_{o_1,r_1}\otimes E_{1,r_{21}}\otimes E_{1,r_{22}})\\&\hspace*{-3.75cm} (E_{1,1}\otimes E_{1,1}\otimes E_{y,z}+E_{1,1}\otimes E_{2,2}\otimes E_{y,z}\\&\hspace*{-3.75cm}+E_{1,1}\otimes E_{3,3}\otimes E_{y,z}+E_{1,1}\otimes E_{4,4}\otimes E_{y,z}\\&\hspace*{-3.75cm}+E_{2,2}\otimes E_{1,1}\otimes E_{y,z} +E_{2,2}\otimes E_{2,2}\otimes E_{y,z}\\&\hspace*{-3.75cm}+E_{2,2}\otimes E_{3,3}\otimes E_{y,z}+E_{2,2}\otimes E_{4,4}\otimes E_{y,z}\\&\hspace*{-3.75cm}+E_{3,3}\otimes E_{1,1}\otimes E_{y,z}+E_{3,3}\otimes E_{2,2}\otimes E_{y,z}\\&\hspace*{-3.75cm}+E_{3,3}\otimes E_{3,3}\otimes E_{y,z}+E_{3,3}\otimes E_{4,4}\otimes E_{y,z}\\&\hspace*{-3.75cm}+E_{4,4}\otimes E_{1,1}\otimes E_{y,z}+E_{4,4}\otimes E_{2,2}\otimes E_{y,z}\\&\hspace*{-3.75cm}+E_{4,4}\otimes E_{3,3}\otimes E_{y,z}+E_{4,4}\otimes E_{4,4}\otimes E_{y,z})
    \end{align*}
    \begin{align*}
        &\hspace*{-4.3cm} =(E_{k,h}E_{o_1,r_1}E_{1,1}\otimes E_{1,r_{21}}E_{1,1}\otimes E_{1,r_{22}}E_{y,z}\\&\hspace*{-3.75cm}+E_{k,h}E_{o_1,r_1}E_{1,1}\otimes E_{1,r_{21}}E_{2,2}\otimes E_{1,r_{22}}E_{y,z}\\&\hspace*{-3.75cm}+E_{k,h}E_{o_1,r_1}E_{1,1}\otimes E_{1,r_{21}}E_{3,3}\otimes E_{1,r_{22}}E_{y,z}\\&\hspace*{-3.75cm}+E_{k,h}E_{o_1,r_1}E_{1,1}\otimes E_{1,r_{21}}E_{4,4}\otimes E_{1,r_{22}}E_{y,z}\\&\hspace*{-3.75cm}+E_{k,h}E_{o_1,r_1}E_{2,2}\otimes E_{1,r_{21}}E_{1,1}\otimes E_{1,r_{22}}E_{y,z} \\&\hspace*{-3.75cm}+E_{k,h}E_{o_1,r_1}E_{2,2}\otimes E_{1,r_{21}}E_{2,2}\otimes E_{1,r_{22}}E_{y,z}\\&\hspace*{-3.75cm}+E_{k,h}E_{o_1,r_1}E_{2,2}\otimes E_{1,r_{21}}E_{3,3}\otimes E_{1,r_{22}}E_{y,z}\\&\hspace*{-3.75cm}+E_{k,h}E_{o_1,r_1}E_{2,2}\otimes E_{1,r_{21}}E_{4,4}\otimes E_{1,r_{22}}E_{y,z}\\&\hspace*{-3.75cm}+E_{k,h}E_{o_1,r_1}E_{3,3}\otimes E_{1,r_{21}}E_{1,1}\otimes E_{1,r_{22}}E_{y,z}\\&\hspace*{-3.75cm}+E_{k,h}E_{o_1,r_1}E_{3,3}\otimes E_{1,r_{21}}E_{2,2}\otimes E_{1,r_{22}}E_{y,z}\\&\hspace*{-3.75cm}+E_{k,h}E_{o_1,r_1}E_{3,3}\otimes E_{1,r_{21}}E_{3,3}\otimes E_{1,r_{22}}E_{y,z}\\&\hspace*{-3.75cm}+E_{k,h}E_{o_1,r_1}E_{3,3}\otimes E_{1,r_{21}}E_{4,4}\otimes E_{1,r_{22}}E_{y,z}\\&\hspace*{-3.75cm}+E_{k,h}E_{o_1,r_1}E_{4,4}\otimes E_{1,r_{21}}E_{1,1}\otimes E_{1,r_{22}}E_{y,z}\\&\hspace*{-3.75cm}+E_{k,h}E_{o_1,r_1}E_{4,4}\otimes E_{1,r_{21}}E_{2,2}\otimes E_{1,r_{22}}E_{y,z}\\&\hspace*{-3.75cm}+E_{k,h}E_{o_1,r_1}E_{4,4}\otimes E_{1,r_{21}}E_{3,3}\otimes E_{1,r_{22}}E_{y,z}\\&\hspace*{-3.75cm}+E_{k,h}E_{o_1,r_1}E_{4,4}\otimes E_{1,r_{21}}E_{4,4}\otimes E_{1,r_{22}}E_{y,z}),
    \end{align*}
    and by the previous consideration, and by letting $r_1=1$, we have
    \begin{align*}
        &\hspace*{-3.7cm}=(E_{k,h}E_{o_1,1}E_{1,1}\otimes E_{1,r_{21}}E_{1,1}\otimes E_{1,r_{22}}E_{y,z}\\&\hspace*{-3.2cm}+E_{k,h}E_{o_1,1}E_{1,1}\otimes E_{1,r_{21}}E_{2,2}\otimes E_{1,r_{22}}E_{y,z}\\&\hspace*{-3.2cm}+E_{k,h}E_{o_1,1}E_{1,1}\otimes E_{1,r_{21}}E_{3,3}\otimes E_{1,r_{22}}E_{y,z}\\&\hspace*{-3.2cm}+E_{k,h}E_{o_1,1}E_{1,1}\otimes E_{1,r_{21}}E_{4,4}\otimes E_{1,r_{22}}E_{y,z}\\&\hspace*{-3.2cm}+E_{k,h}E_{o_1,1}E_{2,2}\otimes E_{1,r_{21}}E_{1,1}\otimes E_{1,r_{22}}E_{y,z} \\&\hspace*{-3.2cm}+E_{k,h}E_{o_1,1}E_{2,2}\otimes E_{1,r_{21}}E_{2,2}\otimes E_{1,r_{22}}E_{y,z}\\&\hspace*{-3.2cm}+E_{k,h}E_{o_1,1}E_{2,2}\otimes E_{1,r_{21}}E_{3,3}\otimes E_{1,r_{22}}E_{y,z}\\&\hspace*{-3.2cm}+E_{k,h}E_{o_1,1}E_{2,2}\otimes E_{1,r_{21}}E_{4,4}\otimes E_{1,r_{22}}E_{y,z}\\&\hspace*{-3.2cm}+E_{k,h}E_{o_1,1}E_{3,3}\otimes E_{1,r_{21}}E_{1,1}\otimes E_{1,r_{22}}E_{y,z}\\&\hspace*{-3.2cm}+E_{k,h}E_{o_1,1}E_{3,3}\otimes E_{1,r_{21}}E_{2,2}\otimes E_{1,r_{22}}E_{y,z}\\&\hspace*{-3.2cm}+E_{k,h}E_{o_1,1}E_{3,3}\otimes E_{1,r_{21}}E_{3,3}\otimes E_{1,r_{22}}E_{y,z}\\&\hspace*{-3.2cm}+E_{k,h}E_{o_1,1}E_{3,3}\otimes E_{1,r_{21}}E_{4,4}\otimes E_{1,r_{22}}E_{y,z}\\&\hspace*{-3.2cm}+E_{k,h}E_{o_1,1}E_{4,4}\otimes E_{1,r_{21}}E_{1,1}\otimes E_{1,r_{22}}E_{y,z}\\&\hspace*{-3.2cm}+E_{k,h}E_{o_1,1}E_{4,4}\otimes E_{1,r_{21}}E_{2,2}\otimes E_{1,r_{22}}E_{y,z}\\&\hspace*{-3.2cm}+E_{k,h}E_{o_1,1}E_{4,4}\otimes E_{1,r_{21}}E_{3,3}\otimes E_{1,r_{22}}E_{y,z}\\&\hspace*{-3.2cm}+E_{k,h}E_{o_1,1}E_{4,4}\otimes E_{1,r_{21}}E_{4,4}\otimes E_{1,r_{22}}E_{y,z}),
    \end{align*}
    resulting the following
    \begin{align}
        &\hspace*{-4cm}=(E_{k,h}E_{o_1,1}\otimes E_{1,r_{21}}E_{1,1}\otimes E_{1,r_{22}}E_{y,z}\notag\\&\hspace*{-3.5cm}+E_{k,h}E_{o_1,1}\otimes E_{1,r_{21}}E_{2,2}\otimes E_{1,r_{22}}E_{y,z}\label{MHGA:2}\\&\hspace*{-3.5cm}+E_{k,h}E_{o_1,1}\otimes E_{1,r_{21}}E_{3,3}\otimes E_{1,r_{22}}E_{y,z}\notag\\&\hspace*{-3.5cm}+E_{k,h}E_{o_1,1}\otimes E_{1,r_{21}}E_{4,4}\otimes E_{1,r_{22}}E_{y,z},\notag
    \end{align}
    and exactly in a same way as in the left hand side, (\ref{MHGA:2}) is either 0 or it will be equal to one of the summations. So, by setting $r_{21}=1$, and using the Sweedler notation, the equation (\ref{MHGA:2}) will be equal to $E_{k,h}E_{o_1,1}\otimes E_{1,1}\otimes E_{1,r_{22}}E_{y,z}$, and by using the Sweedler notation it is easy to see that the two equations (\ref{MHGA:1}) and (\ref{MHGA:2}), respectively resembling the left and right sides of the co-associativity condition (\ref{Equ:Coa:1}) will be equal, and in order for $\Delta$ to be a comultiplication, we also need $\Delta(a)(1\otimes b)$ and $(a\otimes 1)\Delta(b)$ belong to $M_4(\mathbb C)\otimes M_4(\mathbb C)\cong M_8(\mathbb C)$, which is clearly satisfied. Hence $\Delta$ is a comultiplication.

    By using the Sweedler notation it is easy to see that the two equations (\ref{MHGA:1}) and (\ref{MHGA:2}), respectively resembling the left and right sides of the co-associativity condition (\ref{Equ:Coa}) are equal, and in order for $\Delta$ to be a comultiplication, we also need $\Delta(a)(1\otimes b)$ and $(a\otimes 1)\Delta(b)$ belong to $M_4(\mathbb C)\otimes M_4(\mathbb C)\cong M_8(\mathbb C)$, which is clearly satisfied. Hence $\Delta$ is a comultiplication.
\end{document}